\documentclass[10pt,a4paper,reqno]{amsart}
\usepackage{amsfonts,amsthm,latexsym,amsmath,amssymb,amscd,amsmath,epsf}

\newtheorem{Theorem}{Theorem}
\newtheorem{Lemma}{Lemma}
\newtheorem{Cor}{Corollary}
\newtheorem{Prop}{Proposition}

\newtheorem{Ex} {Example}
\newtheorem{Prob} {Problem}

\newtheorem{Conj}{Conjecture}
\newtheorem{theo+}           {Theorem}
\newtheorem{prop+}           {Proposition}
\newtheorem{coro+}           {Corollary}
\newtheorem{lemm+}           {Lemma}

\theoremstyle{definition}
\newtheorem{defi+}           {Definition}

\newtheorem{not+}            {Notation} 

\theoremstyle{remark}

\newtheorem{rema+}           {Remark}

\newenvironment{definition}{\begin{defi+}}{\end{defi+}}

\newcommand{\al}{\alpha}

\newcommand{\la}{\lambda}

\newcommand {\bC} {\mathbb {C}}
\newcommand {\bR} {\mathbb R}
\newcommand {\bZ} {\mathbb Z}

\newcommand {\Z} {\mathcal Z}
\newcommand {\DD} {\mathbf D}
\newcommand {\bCP} {\mathbb {CP}^1}

\newcommand {\Ga} {\Gamma}
\newcommand {\ga} {\gamma}
\newcommand {\eps} {\epsilon}
\newcommand {\de} {\delta}
\newcommand {\supp} {\mathrm supp~}

\begin{document}
             \numberwithin{equation}{section}

             \title[Homogenized spectral problems]{Homogenized spectral problems for exactly solvable operators: asymptotics of \\ polynomial eigenfunctions}

	  \author[J.~Borcea]{Julius  Borcea$^*$}
             \author[R.~B\o gvad]{Rikard  B\o gvad}
\author[B.~Shapiro]{Boris Shapiro}

\thanks{Communicated by M.~Kashiwara. Received March 15, 2008. Revised 
July 31, 2008}
\thanks{2000 Mathematics Subject Classification(s): 30C15, 31A35, 34E05}
\thanks{Department of Mathematics, Stockholm University, SE-106 91 Stockholm,
            Sweden}
\thanks{E-mail addresses: julius@math.su.se, rikard@math.su.se, 
shapiro@math.su.se}
\thanks{$^*$ Corresponding author}

\begin{abstract}
Consider a homogenized spectral pencil of exactly solvable linear 
differential operators $T_{\la}=\sum_{i=0}^k Q_{i}(z)
\la^{k-i}\frac {d^i}{dz^i}$,
where each $Q_{i}(z)$ is a polynomial of degree at most
$i$ and $\la$ is the spectral parameter.
We show that under mild nondegeneracy assumptions  for all
sufficiently large positive integers $n$ there exist
          exactly $k$ distinct values
$\la_{n,j}$, $1\le j\le k$, of the
spectral parameter $\la$ such that the operator
$T_{\la}$ has a polynomial eigenfunction
$p_{n,j}(z)$ of degree $n$. These 
eigenfunctions split into $k$ different families according to the
asymptotic behavior of their eigenvalues. We conjecture and prove sequential 
versions of three fundamental properties: the limits
$\Psi_{j}(z)=\lim_{n\to\infty} \frac
{p_{n,j}'(z)}{\la_{n,j}p_{n,j}(z)}$ exist, are 
analytic and satisfy  the
algebraic equation
$\sum_{i=0}^k Q_{i}(z) \Psi_{j}^i(z)=0$ almost everywhere in $\bCP$. As a 
consequence we obtain a
class of algebraic functions possessing a branch
   near $\infty\in \bCP$ which is representable as the Cauchy
transform of  a compactly supported probability measure.
\end{abstract}

\maketitle

\tableofcontents

\section{Introduction and main results}

In this paper we study the 
properties of asymptotic root-counting measures for families of (nonstandard) polynomial
eigenfunctions of exactly solvable  linear differential operators. 
Using this  information  we  describe a class of
algebraic functions possessing a branch
   near $\infty\in \bCP$ which is representable (up to a constant
   factor) as the Cauchy
transform of a compactly supported probability measure.


	 \begin {not+}
	A linear
ordinary differential
             operator $T:=\sum_{i=1}^{k}a_{i}(z)\frac{d^i}{dz^i}$ is called
             {\em exactly
             solvable} or {\em ES} for short if $T$ (nonstrictly) preserves the infinite
             flag $\mathcal P_{0}\subset \mathcal P_{1}
             \subset \mathcal P_{2}\subset \ldots$, where $\mathcal 
P_{i}$ denotes
             the
             linear space of  polynomials in $z$ of degree at most 
$i$, see \cite
{Tu}.

             The well-known classification theorem of S.~Bochner, see \cite
             {Bo} and \cite {Kr}, states that each coefficient
$a_{i}(z)$ of
             an exactly
             solvable $T$ is a polynomial of degree at most $i$.

             $T$ is called {\em strictly exactly solvable} if the flag
             $\mathcal P_{0}\subset \mathcal P_{1}
             \subset \mathcal P_{2}\subset \ldots$ is strictly preserved. One can show that the
             coefficients $a_{i}(z):=\sum_{j=0}^ia_{i,j}z^j$ of a strictly exactly solvable $T$ satisfy the
             condition that the equation $\sum a_{i,i}t(t-1)\ldots(t-i+1)=0$ has
no positive integer solutions, see Lemma \ref{lm:discr} below.

             Finally, let $ES_{k}$ denote the linear space of all $ES$-operators
of order at
             most $k$.
\end {not+}

             Consider a spectral pencil of $ES_k$-operators
            $T_{\la}:=\sum_{i=0}^{k}a_{i}(z,\la)\frac{d^i}{dz^i}$, i.e., a
            parametrized curve  $T:\bC\to ES_{k}$.
              Each value of the parameter $\la$ for which the
            equation $T_{\la}p(z)=0$ has a polynomial solution
            is called a {\em (generalized) eigenvalue} and the
            corresponding polynomial solution is called a {\em (generalized)
            eigenpolynomial}.

            The main problem that we address in this paper is as follows.

            \begin{Prob} Given a
            spectral pencil $T_{\la}$ describe the asymptotics of its
eigenvalues and
            the asymptotics of the root distribution of the corresponding
             eigenpolynomials.
             \label{pr:2}
                \end{Prob}
Motivated by the necessities of the   asymptotic theory of linear ordinary differential equations,
see e.g.~\cite[Ch.~5]{Fe},  we  concentrate below on the fundamental special
             case of  {\em homogenized spectral pencils}, i.e., rational
             normal curves in $ES_{k}$ of the form
             \begin{equation}
             T_{\la}=\sum_{i=0}^kQ_{i}(z)\la^{k-i}\frac{d^i}{dz^i},
             \label{eq:diff}
             \end{equation}
             where each $Q_{i}(z)$ is a polynomial of degree at most $i$.
Consider the algebraic curve $\Ga$  given by the
          equation
          \begin{equation}
          \sum_{i=0}^k Q_{i}(z)w^i=0,
          \label{eq:basic}
          \end{equation}
          where the polynomials
          $Q_{i}(z)=\sum_{j=0}^{i}a_{i,j}z^j$ are
          the same as in (\ref{eq:diff}).
          Given a curve $\Ga$ as in (\ref{eq:basic}) and the pencil $T_{\la}$
          as in (\ref{eq:diff}) with  the same coefficients $Q_{i}(z)$, $0\le i\le k$,  we
          call $\Ga$  the {\em plane curve associated with $T_{\la}$} and we
          say that $T_{\la}$ is the {\em spectral pencil associated with $\Ga$}.

          The curve $\Ga$ and its associated pencil $T_{\la}$ are called of
{\em general type} if the
          following two nondegeneracy requirements are satisfied:
\begin{itemize}
          \item[(i)] $\deg Q_{k}(z)=k$ (i.e., $a_{k,k}\neq 0$),
\item[(ii)] no two roots of the (characteristic) equation
          \begin{equation}
          a_{k,k}+a_{k-1,k-1}t+\ldots+a_{0,0}t^k=0
          \label{eq:char}
          \end{equation}
          lie on a line through the origin (in particular, $0$ is
          not a root of (\ref{eq:char})).
\end{itemize}


          The first statement of the paper is as follows.

             \begin{Prop} If the characteristic equation (\ref{eq:char}) has
	$k$ distinct solutions $\al_{1},\al_{2},$ $\ldots,\al_{k}$ and satisfies the preceding nondegeneracy assumptions (in particular, these imply that $a_{0,0}\neq 0$ and $a_{k,k}\neq 0$) then
\begin{itemize}
            \item[(i)]  for all
             sufficiently large $n$
there exist exactly $k$ distinct eigenvalues
$\la_{n,j}$, $j=1,\ldots,k$,  such that the associated spectral pencil
$T_{\la}$ has a polynomial eigenfunction
$p_{n,j}(z)$ of degree exactly $n$,
\item[(ii)] the eigenvalues $\la_{n,j}$ split into $k$ distinct
families labeled by the roots of (\ref{eq:char})
such that the eigenvalues in the $j$-th family satisfy
              $$\lim_{n\to \infty}\frac{\la_{n,j}}{n}=\al_{j},\quad j=1,\ldots,k.$$
\end{itemize}
\label{lm:basic}
               \end{Prop}

\begin {Theorem} In the notation of Proposition \ref{lm:basic} 
for any pencil $T_{\la}$ of general type and every $j=1,\ldots,k $ 
there exists a subsequence $\{n_{i,j}\}$, $i=1,2,\ldots,$ such that the limits
$$\Psi_{j}(z):=\lim_{i\to\infty} \frac
{p_{n_{i,j}}'(z)}{\la_{n_{i,j}}p_{n_{i,j}}(z)},\quad j=1,\ldots,k,$$ exist almost everywhere in $\bC$ and are analytic functions in some
neighborhood of  $\infty$. Each $\Psi_{j}(z)$ 
satisfies equation (\ref{eq:basic}), i.e., 
$\sum_{i=0}^k Q_{i}(z) \Psi_{j}^i(z)=0$ almost everywhere in $\bC$, and the
functions
$\Psi_{1}(z),\ldots,\Psi_{k}(z)$ are independent sections of $\Ga$ 
considered as a branched covering
over $\bCP$ in a sufficiently small neighborhood of $\infty$. 
\label{th:main}
\end{Theorem} 

As we explain in \S \ref{sec:5}, a key ingredient in the proof of Theorem~\ref{th:main} is the following localization result for  the roots of the above eigenpolynomials. 

\medskip 
\begin {Theorem} For any general
type pencil $T_{\la}$ all the roots of all its polynomial 
eigenfunctions $p_{n,j}(z)$ lie in a certain disk in $\bC$ centered at 
the origin. 
\label{th:c}
\end{Theorem} 
The  sketch  of the  proof of this fundamental result is as follows.  We convert the differential equation satisfied by the eigenpolynomials into a non-linear Riccati type equation for the logarithmic derivatives of the eigenpolynomials. This equation is then solved recursively and the formal solution is shown to be analytic in a neigh\-borhood of $\infty$. This shows that the zeros of the eigenpolynomials (coinciding with the  poles of the logarithmic derivatives) must lie in some compact subset of $\bC$.

Moreover, we conjecture that for each $j$ the above convergence result holds in fact for the whole corresponding sequence of eigenpolynomials.

\begin{Conj} In the notation of Theorem~\ref{th:main}, for every $j=1,\ldots,k$ the limit  
    $$\Psi_{j}(z)=\lim_{n\to\infty} \frac
{p_{n,j}'(z)}{\la_{n,j}p_{n,j}(z)}$$ exists and has all the properties stated in 
Theorem~\ref{th:main} almost everywhere in $\bCP$.
\label{th:ae}
\end{Conj}

We emphasize the fact that  the {\em proof} of Theorem~\ref{th:main} actually shows that Conjecture~\ref{th:ae} is valid  (at least) in a neighborhood of infinity.   

In order to formulate our further results and reinterpret
Theorems~\ref{th:main}--\ref{th:c} and Conjecture~\ref{th:ae} we need some notation and 
several basic notions of potential theory.

\begin{not+}
 The function $L_{n,j}(z)=\frac
{p_{n,j}'(z)}{\la_{n,j}p_{n,j}(z)}$
is called the {\em normalized logarithmic derivative} of the
polynomial $p_{n,j}(z)$. The limit $\Psi_{j}(z)=\lim_{n\to\infty}
L_{n,j}(z)$ (if it exists) will be referred to as
the {\em asymptotic (normalized) logarithmic derivate} of the family
$\{p_{n,j}(z)\}$.

 The {\em root-counting measure} $\mu_{P}$ of a given
polynomial $P(z)$ of degree $m$
          is the  finite probability measure obtained by placing the mass
$\frac {1} {m}$ at
          every root of $P(z)$. (If some root is multiple we place at this point
          the mass equal to its multiplicity divided by $m$.) Given a
          sequence $\{P_{m}(z)\}$ of polynomials we call the limit
          $\mu=\lim_{m\to \infty} \mu_{P_{m}}$ (if it
           exists in the sense of the weak
           convergence of functionals)  {\em  the asymptotic root-counting
measure} of the
           sequence $\{P_{m}(z)\}$.

           If $\mu$ exists  and is compactly supported it is also a probability
           measure and its support $\supp \mu$ is the limit (in the
           set-theoretic sense) of the sequence $\{\Z_{P_{m}}\}$ of the zero
loci to $\{P_{m}(z)\}$.

 The Cauchy transform of a complex-valued
compactly supported finite
          measure $\rho$  is given
          by $$C_{\rho}(z)=\int_{\bC}\frac{d \rho(\xi)}{z-\xi}.$$
          Note that $C_{\rho}(z)$ is defined at each point $z$ for which
          the Newtonian potential
          $$U_{\vert \rho \vert}(z)=\int_{\bC}\frac{d\vert
          \rho\vert(\zeta)}{\vert \zeta - z\vert}$$
          is finite.
It is easy to see that $C_{\rho}(z)$ exists a.e. and that the
original measure $\rho$ can be restored
        from its Cauchy transform by the formula 
        $$\rho=\frac {1}{\pi}\frac {\partial C_{\rho}(z)}{\partial \bar
        z},$$
        where $\frac 1 \pi\frac {\partial
          C_{\rho}(z)}{\partial \bar z}$ is considered as a distribution, see
          e.g.~\cite[Ch.~2]{Ga}.

 Note also that the Cauchy transform
$C_{\mu_{P}}(z)$ of the
root-counting measure $\mu_{P}$
             of  a given degree $m$ polynomial $P(z)$  coincides with $\frac
{P'}{mP}$.
 \end{not+}

A reinterpretation of Theorem \ref{th:main} in the above terms is as
follows.

\begin{Prop} In the notation of Theorem \ref{th:main}, for each
$j=1,\ldots,k$ there exists a subsequence $\{n_{i,j}\}$, $i=1,2,\ldots$,  such that the asymptotic
root-counting measure $\mu_{j}$ of the 
family $\{p_{n_{i,j}}(z)\}_i$ exists and has compact support with 
vanishing Lebesgue area.
The Cauchy transform of $\mu_{j}$ coincides with $\al_{j}\Psi_{j}$ almost everywhere in $\bC$.
\label{pr:main}
\end{Prop}

As an illustration we present below some numerical results
showing a rather complicated behavior of the zero loci $\Z_{p_{n,j}}$
for $j=1,\ldots,k$. Recall that $\supp \mu_{j}=\lim_{n\to \infty}\Z_{p_{n,j}}$ and 
note that there are slight differences between the scalings used in the four pictures shown in Fig.~\ref{fig1}.

\begin{figure}[!htb]
\centerline{\hbox{\epsfysize=5cm\epsfbox{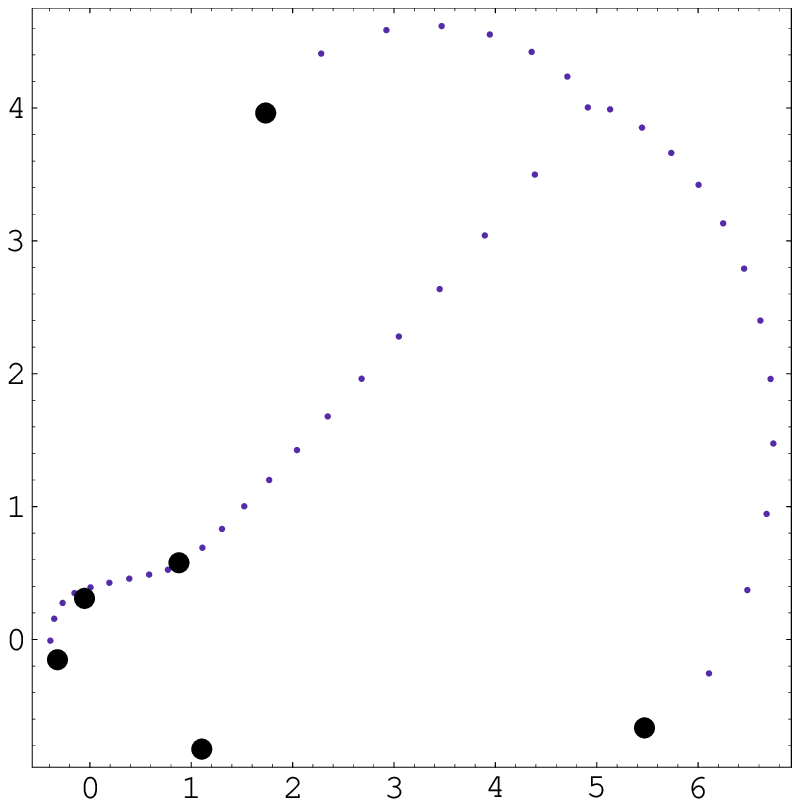}}
\hskip1cm\hbox{\epsfysize=5cm\epsfbox{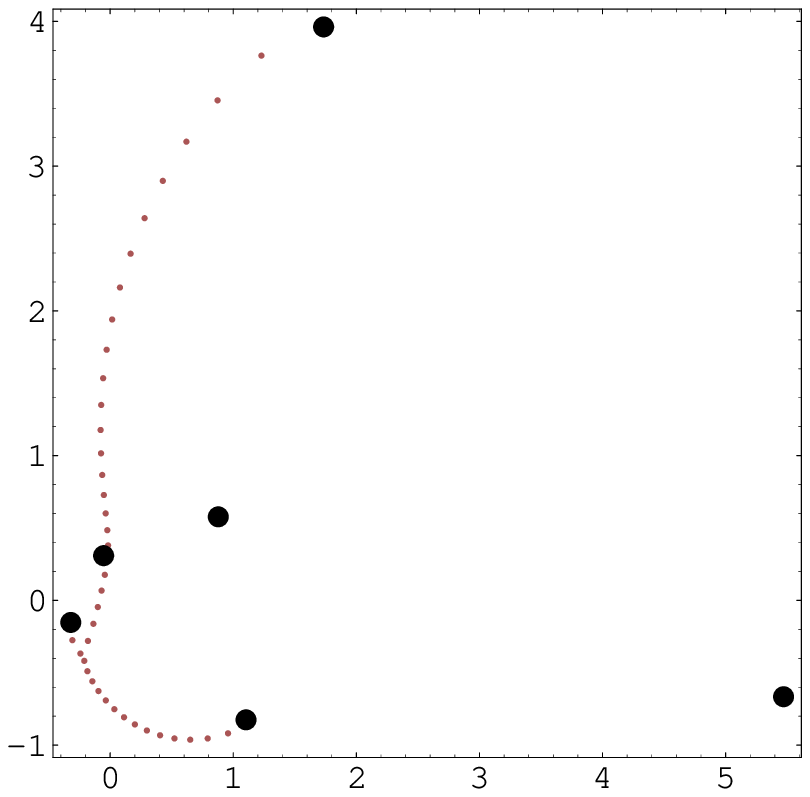}}}
\centerline{\hbox{\epsfysize=5cm\epsfbox{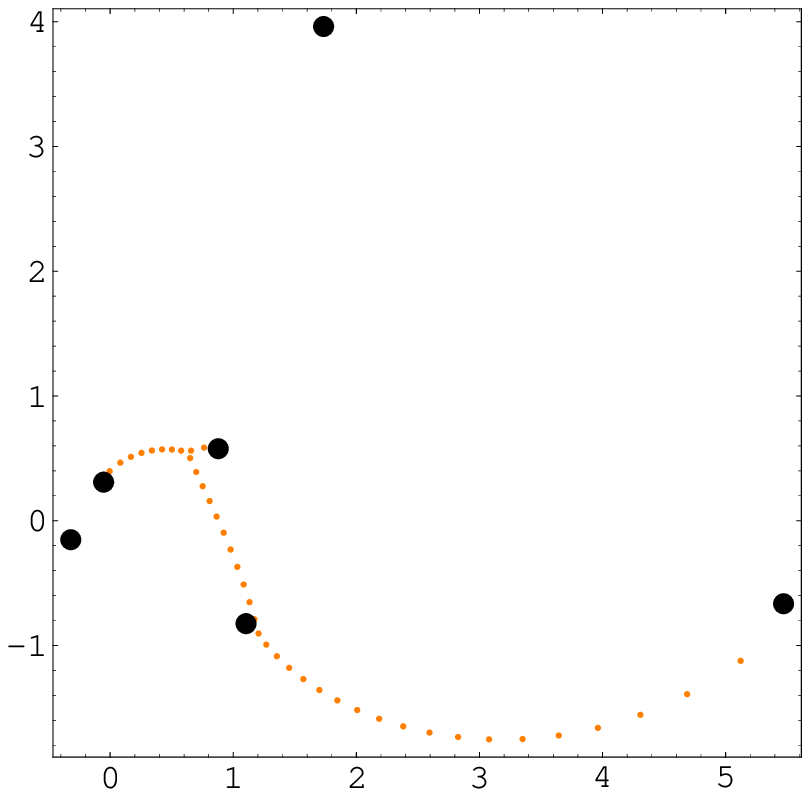}}
\hskip1cm\hbox{\epsfysize=5cm\epsfbox{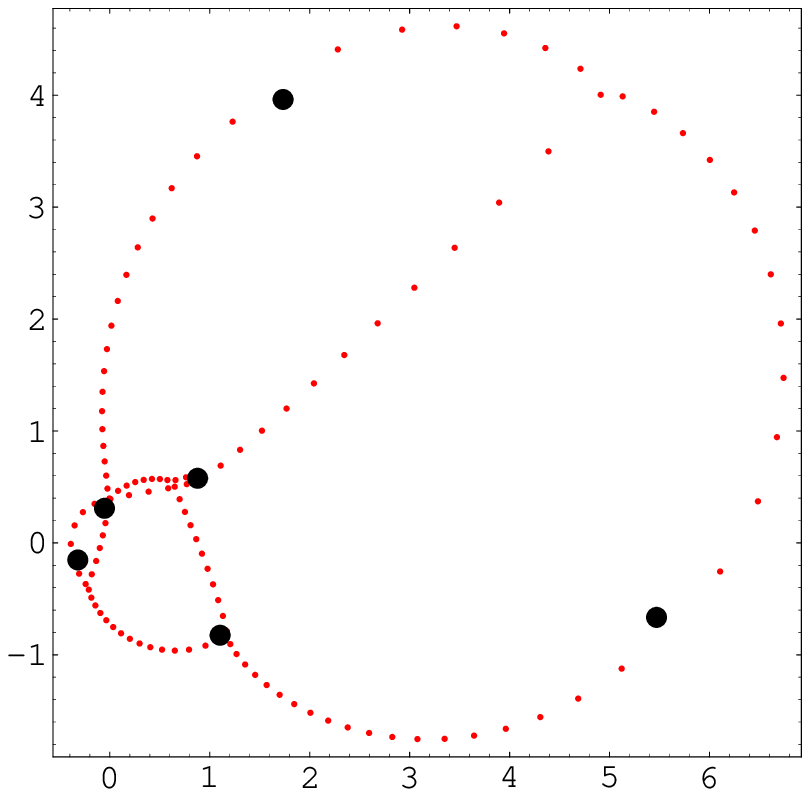}}}
\caption{Three root-counting measures for a third order homogenized spectral pencil. }
\label{fig1}
\end{figure}

\medskip \noindent
{\bf Explanations to Fig.~\ref{fig1}.} The two pictures in the upper row
and the left picture in the bottom row show the roots of three
eigenpolynomials of degree $55$ for the (ad hoc chosen) homogenized spectral
pencil
$$T_{\la}=\left(z^3-(5+2I)z^2+(4+2I)z\right)\frac{d^3}{dz^3}+\la(z^2+Iz+2)\frac{d^2}{dz^2}
+\la^2\frac {1}{5}(z-2+I)\frac{d}{dz}+\la^3$$
consisting of $ES_3$-operators.
The remaining
picture shows the union of the roots of all three polynomials. The six fat
points on all pictures are the branching points of the associated algebraic
curve $\Ga$  given by
$$ \left(z^3-(5+2I)z^2+(4+2I)z\right)w^3+(z^2+Iz+2)w^2+
\frac {1}{5}(z-2+I)w+1=0$$ and considered as the branched
covering over the $z$-plane. Numerical comparisons of the eigenfunctions of 
different degrees show that the above three
root distributions already give a very good approximation of the 
three corresponding limiting 
probability measures
$\mu_{1},\mu_{2},\mu_{3}$ whose Cauchy transforms generate
(after appropriate scalings, see
Theorem~\ref{th:main}) the
three branches of $\Ga$ near $\infty$. Note that all three supports
end only at the six branching points of $\Ga$.  

\medskip
As far as the measures $\mu_{j}$ are concerned, in this paper we establish  
just some of their most important properties. 
To prove these we use two facts about $C_{\mu_j}(z)$. First,  that $\alpha_j^{-1}C_{\mu_j}(z)$ satisfies the algebraic equation  (\ref{eq:basic}) and hence can be locally written as a finite sum $\sum_{i=1} ^r\chi_i\alpha_i^{-1}\Psi_i$, where the $\chi_i$'s are characteristic functions of certain sets. Second, that
  $$\frac {\partial C_{\mu_j}(z)}{\partial \bar
        z}\geq 0$$
        in the sense of distributions.  Using these facts we can develop a natural algebraic geometric setting and build on complex analytic techniques based on the main results of \cite{BB} in order to prove the next two theorems.

\begin{Theorem} For each pencil  $T_\la$ of general type there exists a real-analytic subset $\Gamma_2$ of ${\bC}$ such that  any limiting measure $\mu_j$ has properties (A)--(C) below in any sufficiently small neighborhood $\Omega(z_{0})$ of any point $z_0\in {\bC}\setminus \Gamma_2$. In what follows $\Gamma$ denotes  the curve associated with the pencil $T_\la$, $\alpha_j\in\bf C$ is as in Proposition~\ref{lm:basic}, and for any branch $\gamma_i$ of $\Ga$ we let $A_i:=\alpha_j\ga_i$ (clearly, $A_i$ depends on $j$).
\begin{itemize}
\item[(A)] 
The support $S$ of the measure $\mu_{j}$ restricted to $\Omega(z_{0})$  is a finite union of smooth curves $S_r,\ r\in J$.
\item[(B)] For each $S_r$ and any $\tilde z\in
S_{r}$ lying in $\Omega(z_{0})$ one can always choose two branches $\ga_1(z)$ and $\ga_{2}(z)$
of $\Ga$ such that the tangent line $l(\tilde z)$ to $S_{r}$  is orthogonal to $\overline{A_{1}(\tilde z)}-\overline{A_2(\tilde z)}$.
\item[(C)] The density
of the measure $\mu_{j}$ at $\tilde z$ equals
$\frac{\vert  {A_{1}(\tilde z})-A_{2}({\tilde z}) \vert ds}{2\pi}$,
where $ds$ is the length element along the curve $S_{r}$.
\end{itemize}
\label{th:rikard}
    \end{Theorem}

    In fact for most pencils we can do better:
    
    \begin{Theorem} For a typical general type pencil $T_\la$ the set $\Gamma_2$ in the preceding theorem is finite.
    \label{th:generic}
    \end{Theorem}

\begin{rema+}
In the case of the usual spectral problem strong results  in this direction were obtained in \cite{BR}. 
\end{rema+}

           We note that Theorem \ref{th:ae} also leads to a partial progress on the
           following intriguing question in potential theory. Let $Alg(i,j)$
           denote the linear space of all polynomials in the variables
           $(y,z)$ of bidegrees at most $(i,j)$, i.e., each $P(y,z)\in
           Alg(i,j)$ has the form $\sum_{l=0}^iP_{l}(z)y^l$, where $\deg
           P_{l}(z)\le j,\;l=1,\ldots,i$. Abusing the notation we identify each
           $P(y,z)\not \equiv 0$ with the algebraic function in the variable
           $z$ defined by the equation $P(y,z)=0$.

           \begin{Prob}\label{pb2} Describe the subset $PAlg(i,j)\subset Alg(i,j)$
           consisting of all algebraic functions which have a branch near
           $\infty$ coinciding with the Cauchy transform of some
           probability measure compactly supported in $\bC$.
		\end{Prob}

	We refer to such algebraic functions as {\em positive Cauchy
transforms}.
	The only case when a complete answer to the above problem seems to
	be obvious is for bidegrees $(1,j)$, i.e., the case of rational
	functions. Namely, a rational function $r(z)=\frac {p(z)}{q(z)}$ is
	a positive Cauchy transform if and only if $r(z)$ is of the form 
$$\sum_{l=1}^j\frac
	{c_{l}}{z-z_{l}}\text{ with }c_{l}>0,1\le l\le j,\, z_{k}\neq z_{l}\text{ for }k\neq l,\,\sum_{l=1}^jc_{l}=1.$$


Let us finally give yet another interpretation of Theorem~\ref{th:main} 
and
           Proposition~\ref{pr:main}.

              \begin{Cor} Each  branch near infinity of an 
algebraic function satisfying (\ref{eq:basic}) is   a
positive Cauchy transform multiplied by an appropriate constant. 
            \label{cor:CT}
	 \end{Cor}

          \medskip
        
     The structure of the paper is as follows.  In \S 2 we study the
          asymptotics of the eigenvalues to (\ref{eq:diff}) and some simple
          properties of the corresponding eigenfunctions as well as the
          defining algebraic curve (\ref{eq:basic}). In \S 3 we solve the
          spectral problem defined by the operator (\ref{eq:diff}) in formal
          power series near $\infty$ using the variable $y:=z^{-1}$. In \S 4
          we show that the power series solution obtained in \S 3 converges
          in some neighborhood of $\infty$. Based on these results we prove  Theorem \ref{th:c}
	  as well as 
	  Theorem~\ref{th:main} and its corollaries in \S\ref{sec:5}. In \S 6
          we prove Theorem~\ref{th:rikard} and complete the proof of 
          Theorem~\ref{th:generic}.   Finally, in \S 7 we propose a number of open problems and 
          conjectures on the asymptotic
          behavior of polynomials functions for linear ordinary differential
          operators and place these in a wider context that encompasses both 
old and new literature on this and related topics.

\medskip
          {\it Acknowledgements.} We wish to thank 
          Hans~Rullg\aa rd  for stimulating
          discussions on these subjects. It  is difficult to underestimate 
the role of
          the excellent paper \cite {BR}
          in the current project. We would also like to thank the anonymous referee 
for his valuable comments.

\section{Basic facts on asymptotics of eigenvalues and polynomial
eigenfunctions}
\label{sec:asymp}

In this section we prove Proposition \ref{lm:basic}, that is, we describe the polynomial solutions of  an $ES$ spectral pencil and we also prove the easy part of
Theorem \ref{th:main} saying that if the limit functions $\Psi_{j}(z)$, $1\le j\le k$, exist locally and have
derivatives of sufficiently high order then they must
satisfy equation (\ref{eq:basic}).

          Denote by  $\mathcal D_{n}\subset ES_{k}$ the subset of
              all $ES$-operators $T$ of order  at most $k$ such that the
equation $Ty=0$ has
             a polynomial solution of degree exactly $n$.

             \begin{Lemma} In the notation of Theorem \ref{th:main} the closure
             $\overline {\mathcal D_{n}}$ of the discriminant  $\mathcal
D_{n}$ is a
             hyperplane in $ES_{k}$
             given by the equation
             \begin{equation}
             \sum_{i=0}^k n(n-1)\ldots(n-i+1) a_{i,i}=0.\label{eq:discr}
             \end{equation}
             \label{lm:discr}
             \end{Lemma}

             \begin{proof} Since any $ES$-operator $T$ (nonstrictly) preserves
             the infinite flag $\mathcal P_{0}\subset \mathcal P_{1}\subset
             \ldots\subset \ldots$ of linear spaces of polynomials of at
most given
             degree it follows that $T$ has an eigenpolynomial of degree exactly $n$ if
             and only if
\begin{itemize}
             \item[(i)] $T$ restricted to $\mathcal P_{n}$ is degenerate,
\item[(ii)] the kernel of $T$ restricted to $\mathcal P_{n}$ intersects
             $\mathcal P_{n}\setminus \mathcal P_{n-1}$.
\end{itemize}
             \noindent
             The closure
$\overline {\mathcal D_{n}}$
             consists of all $T$ having an eigenpolynomial of degree 
at most $n$.
             Now $T$ is upper triangular in the standard 
             monomial basis and the $j$-th diagonal entry of $T$ equals
             $\sum_{i=0}^k j(j-1)\ldots(j-i+1) a_{i,i}$. Therefore 
             $\overline {\mathcal D_{n}}$ is given by equation
(\ref{eq:discr}).
\end{proof}

	\noindent
	\begin{rema+} If $T\in\overline {\mathcal D_{n}}$ but $T\notin\overline {\mathcal D_{l}}$ for $0\le l<n$ then $T$
	has a polynomial solution of degree exactly $n$. Otherwise $T$ has at least  a polynomial solution of degree equal to
	$\min\{l\in\{0,\ldots,n-1\}: T\in \overline {\mathcal D_{l}}\}$ (and
	probably  some other polynomial solutions as well).
	\end{rema+} 

	    We now turn to Proposition \ref{lm:basic}.

	\begin{proof} By Lemma \ref{lm:discr} the homogenized spectral
	pencil $T_{\la}$ has a polynomial solution of degree at most $n$
	if 
	\begin{equation}
\sum_{i=0}^kn(n-1)\ldots (n-i+1)a_{i,i}\la^{k-i}=0.
\label{eq:spect}
\end{equation}
This equation is  of degree exactly  $k$ in $\la$ since $a_{0,0}\neq 0$. Considering $n$
 for the moment as a complex variable, the equation defines an algebraic curve in $\bC^2$ with $k$ branches over $\bC$. The behavior of the branches at infinity may be found by substituting
          $\widetilde \la=\la/n$ and dividing the left-hand side of (\ref{eq:spect})
          by $n^k$, which  gives 
         \begin{equation}
         \label{eq:spect2}
         \sum_{i=0}^k 1 \cdot\frac {(n-1)}{n}\cdot\frac {(n-2)}{n}\cdot\ldots
          \cdot\frac{(n-i+1)}{n}\cdot a_{i,i}\cdot 
\widetilde\la^{k-i}=0.
\end{equation}
          If $n\to\infty$ the latter family of equations tends coefficientwise
          to equation (\ref{eq:char}). Since the latter has exactly $k$ different solutions, this is true for  (\ref{eq:spect2}) as well if $n\in N $ for some simply connected neighborhood $N\subset \bCP$ of infinity. Hence in $N$ there are $k$ different branches $\widetilde\la_{n,1},\ldots,\widetilde\la_{n,k}$ of the algebraic function $\widetilde\la$ defined as a function of the complex variable $n$ by equation (\ref{eq:spect2}), and corresponding branches $\la_{n,1},\ldots,\la_{n,k}$ of the algebraic function $\la$. 
          Then clearly one has
          $$\lim_{n\to \infty }\la_{n,j}/n =\lim_{n\to
\infty}\widetilde\la_{n,j}=\al_{j},\quad j=1,\ldots,k.$$ 
This completes the proof of part (ii).
To prove further that for each branch 
$\la_{n,j}$, $j=1,\ldots,k$,  the associated spectral pencil
$T_{\la}$ has a polynomial eigenfunction
$p_{n,j}(z)$ of degree {\it exactly} $n$, we need to show that for a $\la$ and $n$ that solve equation (\ref{eq:spect}) there is no $m\neq n$ such that $\la$ solves equation (\ref{eq:spect}) for that $m$. This follows if we first prove that there are no solutions $n\neq m$ to $\la_{n,j}=\la_{m,j}$, $j=1,\ldots,k$, and secondly that if $j_1\neq j_2$ then there are no solutions to $\la_{n,j_1}=\la_{m,j_2}$ (both statements in a possibly shrunk neighborhood of infinity). The first statement is an immediate consequence of the assumptions since they imply that $\al_j\neq 0$, $j=1,\ldots,k$, and thus for each $j$ the function 
$n\mapsto \la_{n,j}$ is one-to-one in a neighborhood of infinity.  For the second statement argue by contradiction as follows. If the statement is false there are unbounded sequences of positive integers $n_i,m_i$, $i=1,2,\ldots$ such that $\la_{n_i,j_1}=\la_{m_i,j_2}$ for all $i$. We may assume (by choosing a subsequence and possibly interchanging $j_1$ and $j_2$) that $\lim_{i\to\infty}m_i/n_i=r\in \bR$. Hence 
$$\al_{j_1}=\lim_{i\to\infty}\la_{n_i,j_1}/n_i=\lim_{i\to\infty}(\la_{m_i,j_2}/m_i)(m_i/n_i)=r\al_{j_2}
$$
and so the arguments of $\al_{j_1}$ and $\al_{j_2 }$ are equal, which contradicts the nondegeneracy assumptions in the proposition. This completes the proof.
\end{proof}

\noindent
\begin{rema+} It is possible and straightforward to calculate the asymptotics for the
eigenvalues even in the case when (\ref{eq:char}) has multiple or
vanishing roots.
\end{rema+}

Let us now prove the easy part of Theorem \ref{th:main}.

            \begin{Prop}
                \label{dlog}
                Let $\{p_{n,j}(z)\}$ be a family of polynomial
eigenfunctions of a homogenized
                spectral pencil $T_{\la}$ with corresponding family of eigenvalues
                $\{\la_{n,j}\}$, i.e., $p_{n,j}(z)$ satisfies the equation
	$T_{\lambda_{n,j}}p_{n,j}(z)=0$.
		Assume that the following holds:
\begin{itemize}
		\item[(i)] $\lim_{n\to \infty}\la_{n,j}=\infty$,
\item[(ii)]  there exists an
		open set $\Omega\subseteq \bCP$ where
	the normalized logarithmic derivatives $L_{n,j}(z)=\frac
{p_{n,j}'(z)}{\la_{n,j}p_{n,j}(z)}$
	are defined for all sufficiently large $n$  and the sequence
	$\{L_{n,j}(z)\}$ converges in $\Omega$ to a 
	function
	$$\Psi_{j}(z):=	\lim_{n\to \infty}L_{n,j}(z),
	$$
\item[(iii)] the $k-1$ first derivatives of the sequence $\{L_{n,j}(z)\}$
	are uniformly bounded in $\Omega$.
\end{itemize}
Then if  $\Psi_{j}(z)$ does not vanish identically  
        it satisfies the equation
	\begin{equation}
	\sum_{i=0}^kQ_{i}(z)\Psi^i_{j}(z)=0.
	\label{algfunc}
	\end{equation}
\end{Prop}

\begin{proof} In order to simplify the notation in this proof let us fix
the value of the
           index $j\in\{1,\ldots,k\}$ and simply drop it.

Note that each $L_{n}(z)=\frac{p_{n}'(z)}{\la_{n}p_{n}(z)}$
is well defined and analytic in any disk $\DD$ free from the zeros of
$p_{n}(z)$.
             Choosing such a disk $\DD$ and an appropriate branch of the
logarithm such that $\log p_{n}(z)$ is defined in $\DD$ let us consider a
primitive function
$M(z)=\lambda^{-1}_{n}\log p_{n}(z)$ which is also well defined and
analytic in $\DD$.

Straightforward calculations give: $e^{\lambda_{n}M(z)}=p_{n}(z),\;$
$p'_{n}(z)=p_{n}(z)\lambda_{n}
L_{n}(z)$, and
$p''_{n}(z)=p_{n}(z)(\lambda^2_{n}L^2_{n}(z)+\lambda_{n}L_{n}'(z))$.
More generally,
$$
\frac{d^i}{dz^i}(p_{n}(z))p_{n}(z)\left(\lambda^i_{n}L^i_{n}(z)+\lambda^{i-1}_{n}F_{i}(L_{n}(z),L'_{n}(z),
\ldots,
L^{(i-1)}_{n}(z))
\right),$$
where the second term
\begin{equation}
	\lambda_{n}^{i-1}F_{i}(L_n,L_n',\ldots,L_n^{(i-1)})
	\label{eq:Fdesc}
\end{equation}
is a polynomial of degree $i-1$ in $\lambda_{n}$. 
The equation $T_{\lambda_{n}}p_{n}(z)=0$ gives us
$$
p_{n}(z)
\left(\sum_{i=0}^{k}Q_{i}(z)\la_{n}^{k-i}\left(\lambda_{n}^iL_{n}^{i}(z)+\lambda_{
n}^{i-1}F_
{i}(L_{n}(z),L'_{n}(z),\ldots,L^{(i-1)}_{n}(z))\right)\right)=0
$$
or equivalently,
\begin{equation}
	\label{Dlog}
	\lambda_{n}^k\sum_{i=0}^{k}Q_{i}(z)\left(L^i_{n}(z)+\lambda_{n}^{-1}F_{i}(L_{n}(z),L
'_{n}(z),
	\ldots,L^{(i-1)}_{n}(z))\right)=0.
\end{equation}
Letting $n\to \infty$ and using  the
boundedness assumption for the first $k-1$ derivatives we get the required
equation (\ref{algfunc}).
\end{proof}

We end this section by establishing an important  property of the curve
$\Ga$ given by (\ref{eq:basic}). Note that
          unless $Q_{k}(x)\equiv 0$ the curve $\Ga$
          is a $k$-sheeted
branched covering of the $z$-plane. We want to describe
the behavior of $\Ga$ at infinity.
Using a change of coordinate $y:=z^{-1}$ we can rewrite equation
(\ref{eq:basic})
as
$$
\sum_{i=0}^{k}Q_{i}(z)w^{i}=\sum_{i=0}^{k}z^{-i}Q_{i}(z)(wz)^{i}=\sum_{i=0}^{k}P_{i}(y)\xi^{i}=0,
$$
where
$$
P_{i}(y)=z^{-i}Q_{i}(z)=\sum_{j=0}^{i}a_{ij}y^{i-j}$$
and $\xi:=wz=w/y$.
Note also that at the point $y=0$ one gets the
reciprocal characteristic equation
\begin{equation}
	\label{lambdadef}
	a_{k,k}\xi^{k}+a_{k-1,k-1}\xi^{k-1}+\ldots+a_{0,0}=0.
	\end{equation}

	\noindent
\begin{rema+}
	The roots $\xi_{1},\ldots,\xi_{k}$ of (\ref{lambdadef}) are
the inverses of the roots
	$\al_{1},\ldots,\al_{k}$, respectively, of (\ref{eq:char}).
	\end{rema+} 

Using the above argument we get the following simple statement.

	\begin{Lemma} If the roots $\xi_{1},\ldots,\xi_{k}$
		of equation (\ref{lambdadef}) are pairwise distinct then there
are $k$	branches $\ga_{i}(z)$, $i=1,\ldots,k$, of the curve $\Ga$ 
		that are well defined in some common neighborhood of $z=\infty$. The
$i$-th	branch $\ga_{i}(z)$ satisfies the normalization condition 
		$
		\lim_{z\to \infty}z\ga_{i}(z)=\xi_{i}.
		$
		\label{lm:gen}
		\end{Lemma}

\section{Solving equation (\ref{Dlog}) formally}\label{sec:3}

This and the following two sections are completely devoted to the proof of 
Theorems~\ref{th:main} and~\ref{th:c}. The sketch of the proof of Theorem~\ref{th:main} is as follows. In Proposition \ref{dlog} of \S 2 we have transformed the linear differential equation for the eigenpolynomials into a non-linear Riccati type equation for the logarithmic derivative. In this section we first analyze closely the terms of this new equation. 
We then describe the recursion scheme for solving this equation formally in a neighborhood of infinity and see how the solutions behave when $\lambda\to\infty$. Finally, in the next section we show that there is a neighborhood of   $z=\infty$ where the formal solutions are analytic and we complete the proofs of Theorems~\ref{th:main} and~\ref{th:c} in \S \ref{sec:5}.

Throughout this section we use the variable $y:=1/z$ near $z=\infty$.

\medskip \noindent
{\bf The differential algebra $A_{z,L}$.}
As the {\em first step} we describe the terms occurring in equation
(\ref{Dlog}) more
precisely. It is convenient to do this in a universal setting using
the following infinitely generated free commutative $\bC$-algebra
(or rather the free differential algebra, cf. \cite{Kaplansky})
$$
A_{z,L}:={\bC}[\lambda,\lambda^{-1},z,z^{-1},L^{(0)},L^{(1)},L^{(2)},\ldots].
$$
This algebra should be thought of as (a universal object) containing the terms in equation (\ref{Dlog}).  Concrete instances can be obtained by specialization. In particular, the $L^{(i)}$'s correspond to the normalized logarithmic derivatives in  (\ref{Dlog}).

Note that the monomials $\lambda^{i}z^{j}L^{I}$ form a  basis of
$A_{z,L}$ considered as a vector space, where for any
multi-index $I=(i_{1},\ldots,i_{r})$
         the symbol $L^{I}$ denotes the product
$$
L^{I}=\prod_{s=1}^{r} L^{(i_{s})}.
$$
It suffices to use multi-indices $I$ that are finite
non-decreasing sequences of non-negative
integers. Denote
the set of all such multi-indices by $FS$. (By definition, $L^{(0)}:=L$.)
Such index sequences may alternatively be
thought of as finite multisets. In particular, we include the empty sequence
$\emptyset $ in $FS$ and
interpret $L^{\emptyset}:=1$. For a given multi-index $I$
define its {\it modulus} by $\vert I\vert=\sum_{s=1}^{r}i_{s}$ and its
        {\it length} by $lng(I)=r$. (By definition, $\vert 
\emptyset\vert:=0$ and
        $lng(\emptyset):=0$).  For a given $I\in FS$ denote by $I_{+}$ the
sequence obtained from
$I$ by discarding all its $0$ elements.

$A_{z,L}$ is equipped with a natural  derivation $D_{z}$ which is a
prototype of $\frac{d}{dz}$. Namely, $D_{z}$ is uniquely defined by the
relations:\quad  $D_{z} *\lambda=0$, $D_{z} *z=1$ and $D_{z}
*L^{(i)}=L^{(i+1)}$. (We
 use the symbol ``$*$'' to denote the action of a differential operator
        as opposed to the product of differential operators. Note that
the ring of
differential operators generated by $A_{z,L}$ and $D_{z}$ acts on $A_{z,L}$.)

\medskip\noindent
{\bf The normalized logarithmic derivative in the universal setting.}
We can next use $A_{z,L}$ to describe the relation between the derivatives of an eigenpolynomial and its normalized logarithmic derivative. This is done by defining a differential $A_{z,L}$-module.
Consider the free rank one $A_{z,L}$-module
	$A_{z,L}e^{M}$, where $e^{M}$ is the generator.  Define
$D_{z}*e^{M}:=\lambda Le^{M}$ and extend the
	action of $D_{z}$ to the whole $A_{z,L}e^{M}$ using the Leibnitz rule. Note that this action intuitively just says that $L$ is the normalized logarithmic derivative $ D_{z}*e^{M}/\lambda e^{M}$ of the generator $e^M$.
The following lemma can be easily proved by induction.
\begin{Lemma}
	\label{firstformal}  In the above notation one has 
	$
	(D_{z})^{i}*e^{M}=R_{i}e^{M},
	$
	where $R_{0}=1$, $R_{1}=\lambda L$ and $R_{i+1}=(\lambda
	L+D_{z})*R_{i}$ for $i\ge 1$. In other words,
	$$R_{i}=(\lambda
	L+D_{z})^{i}*1=(\lambda
	L+D_{z})^{i-1}*\lambda L,\quad i\ge 1.
	$$
	\end{Lemma}

The $R_i$ considered as polynomials $R_i(\lambda,L^{(0)},\ldots)$ have (by universality) the following property. 

\begin{Lemma}
	\label{firstuniversal} Let $g={\lambda}^{-1}{D\log f}$, where $f$ is a non-vanishing analytic function in an open subset $\Omega\subset {\bC}$ and $\lambda\in{\bC}\setminus\{0\}$ (so $g$ is the logarithmic derivative of $f$ normalized  with respect to $\lambda$).
	Then in the above notation one has 
	 $$f^{(i)}=R_i(\lambda,g,g^{(1)},\ldots.)f,\quad i\ge 1.
	 $$	 	\end{Lemma}
	
\medskip\noindent
{\bf The differential algebra $A_{z,L}$ at $\infty $.}
Next we  rewrite (\ref{Dlog}) using the variable $y=z^{-1}$.
Note that if $p(z)$ is a non-constant polynomial then  for any $\kappa\neq
0$ the logarithmic derivative $L(z)=p'(z)/\kappa
p(z)$,
        rewritten using the variable $y=\frac{1}{z}$, has a simple zero at
        $y=0$.
Hence we should look for solutions  of equation
(\ref{Dlog}) in the form $L(z)=yN(y)$.

First we describe $A_{z,L}$ near $z=\infty$. For this we define an analogous free
commutative algebra
$$
B_{y,N}:={\bC}[\lambda,\lambda^{-1},y,y^{-1},N^{(0)},N^{(1)},\ldots].
$$
As above, it has  a natural derivation $D_{y}$ which is a prototype
of $\frac{d}{dy}$
satisfying the relations:
$\,D_{y}*\lambda=0$, $D_{y}*y=1$ 
and $D_{y}*N^{(i)}=N^{(i+1)}$. Note that
$\frac{d}{d z}=-y^2\frac{d}{d y}$  and recall that $L:=L^{(0)}$ and
$N:=N^{(0)}$. Define an injection $\Theta:A_{z,L}\to B_{y,N} $  of algebras
determined by 
$$\Theta(\lambda)=\lambda, \,\Theta(z)=y^{-1},\, 
           \Theta(L)=yN,\,\Theta(L^{(i)})=(-y^2D_{y})^{i}*yN.$$

The following lemma describes the
connection between $A_{z,L}$ and $B_{y,N}$.

\begin{Lemma}
           \label{2formal}
           The injection $\Theta$ has the property that for any  $a\in A_{z,L}$
           one has
           $$
           \Theta(D_{z}^{i}*a)=(-y^2D_{y})^{i}*\Theta(a).
           $$
           \end{Lemma}

        \begin{proof} By definition the above formula is valid in the
case $i=1$ for
           all the generators of $A_{z,L}$. Since $D_{z}$ and $-y^2D_{y}$
        are derivations the formula then works for all elements in this
        algebra and $i=1$.
        Simple induction  shows that the above formula is valid even for all
        $i>1$.
            \end{proof}

            \medskip\noindent
            {\bf Main lemma on $R_i$.}
The preceding lemmas imply the following relation:
\begin{equation}
           \lambda^{-i} y^{-i}\Theta
(R_{i})=y^{-i}(yN-y^2\lambda^{-1}D_{y})^{i} *1=y^{-i}(yN-y^2\lambda^{-1}D_{y})^{i-1}y *N.
\label{eq:theta}
\end{equation}

We need to know which monomial terms of the form
$\lambda^{\alpha_{1}}y^{\alpha_{2}}N^I$ occur in this
equation.
   These monomials are contained in the
subalgebra  $B_{0}\subset B_{y,N}$ defined as 
$$
B_{0}:={\bC}[\lambda^{-1},y,N,\lambda^{-1}yN^{(1)},\ldots,(\lambda^{-1}y)^{j}N^{
(j)},\ldots],
$$
see Lemma~\ref{lm:form} below. Define a (necessarily two-sided, by commutativity) 
ideal $J\subset B_{0}$ as follows:
$$J:=<\lambda^{-1},(\lambda^{-1}y)^{2}N^{(1)}N^{(1)},\ldots,
(\lambda^{-1}y)^{j_{1}+j_{2}}N^{(j_{1})}N^{(j_{2})}\ldots>,$$
where $j_{1}\ge 1,\; j_{2}\ge 1$. (As we will see the ideal $J$
contains all the terms
that do not influence the asymptotic behavior of the equation
(\ref{Dlog}).)

One can easily show that  the set of all monomials  of the form
$$(\lambda^{-1})^{\alpha_{1}}y^{\alpha_{2}}(\lambda^{-1}y)^{\vert I\vert}N^I,$$
where $\alpha_{i}\in {\bZ_{\geq 0}},\ i=1,2$, and  $I\in FS$
constitutes a basis of $B_{0}$ as a vector space.
Such a monomial belongs to $J$
if and only if either $\alpha_{1}\geq 1$ or
$lng(I_{+})\geq 2$.

        \begin{Lemma}\label{lem-6}
For all $i\geq 0$ the following identity  between
        elements in $B_{0}$ holds:
        \begin{equation}
\lambda^{-i} y^{-i}\Theta
(R_{i})=N^{i}+\sum_{j=1}^{i-1}\binom{i}{j+1}
(-1)^{j}(\lambda^{-1}y)^{j}N^{i-1-j}N^{(
j)}+h
\label{eq:form}
\end{equation}
for some unique $h\in J$. Moreover, $if i\geq 1$, the exponents of all
non-vanishing monomials
$(\lambda^{-1})^{\alpha_{1}}y^{\alpha_{2}}(\lambda^{-1}y)^{\vert I\vert}N^I$ 
occurring in the right-hand side  of (\ref{eq:form}) satisfy the
inequality 
        $\alpha_{1}+\vert I\vert\leq i-1$.
    \label{lm:form}    
	\end{Lemma}
\begin{proof}
        The case $i=0$ is trivial so we assume that $i\geq 1$. 

\medskip\noindent
{\em Claim.} The differential operator
$
L_{i}:=y^{-(i+1)}(yN-y^2\lambda^{-1}D_{y})^{i  } y
$
preserves both $B_{0}$ and $J$ and satisfies the relation
\begin{equation}
L_{i}*N=N^{i+1}+\sum_{j=1}^{i}\binom{i+1}{j+1}(-1)^{j}N^{i-j}
(\lambda^{-1}y)^{j}N^{(j)}\;(\mathrm{ mod\;} J),\label{lm:form1}\end{equation}
where $(\mathrm{ mod\;} J)$ means that the expression is considered modulo the ideal
$J\subset B_{0}$.

\medskip

The lemma immediately follows from this claim. Indeed, note that by Lemma \ref{firstformal} one has 
\begin{equation}
           \lambda^{-i} y^{-i}\Theta
           (R_{i})   =        y^{-i}(yN-y^2\lambda^{-1}D_{y})^{i-1}y *N =L_{i-1}*N 
           \label{eq:k-1}
\end{equation}
and so \eqref{eq:form} is a consequence of \eqref{eq:k-1} and \eqref{lm:form1}. It thus remains to prove the three assertions in the above claim, which we do below by induction. 
           As the base of induction note that since $L_{0}=1$ it is obvious that $L_{0}$ 
	preserves both $B_{0}$ and $J$ and that it satisfies (\ref{lm:form1}).

\medskip	\noindent{\bf Induction steps.}
We first show inductively that all operators $L_i$ preserve both $B_{0}$ and $J$ and then using this we check formula \eqref{lm:form1} again by induction. Now
the following identity is immediate:
$$L_{i+1}*v=y^{-(i+2)}(yN-y^2\lambda^{-1}D_{y})y^{i+1}L_{i}*v=$$
   \begin{equation}\label{lm:ind2}
=(N-(i+1)\lambda^{-1}
)L_{i}*v-y\lambda^{-1}D_{y}(L_{i}*v),\quad v\in B_0.
\end{equation}
Hence in order to check that $B_{0}$ and $J$ are preserved by $L_{i+1}$
it suffices (using the induction assumption) to
show that both $(N-(i+1)\lambda^{-1})$ and $y\lambda^{-1}D_{y}$
preserve $B_{0}$
and $J$.
Since $N$ and $\lambda^{-1}$ are contained
in $B_{0}$, it is clear that they preserve both $B_{0}$ and  $J$, so it is enough to verify that the same holds for $\lambda^{-1}yD_{y}$. Let us first show that the latter operator preserves the ideal $J$, under the assumption that it preserves $B_0$. For this we note that an arbitrary
element $h\in J$ may be written as $h=\sum b_{i}h_{i}$, where $h_{i}$ are  the generators given in
the definition of $J$  and $b_{i}\in B_{0}$. Thus it is enough to prove that
$y\lambda^{-1}D_{y}*b_ih_i\in J$. As we will now explain, this property follows from the identity
$$
y\lambda^{-1}D_{y}*b_{i}h_{i}=y\lambda^{-1}h_{i}(D_{y}*b_{i})+y\lambda^{-1}b_{i}
(D_{y}*h_{i}).
$$
Indeed, we claim that its right-hand side belongs to $J$. To show this note that the first term
         clearly belongs to $J$, so it suffices to
check that $y\lambda^{-1}(D_{y}*h_{i})\in J$. Now if $h_{i}=\lambda^{-1}$ then this is again obvious. For the other generators we simply use
the identity
\begin{eqnarray*}
y\lambda^{-1}D_{y}*(\lambda^{-1}y)^{j_{1}+j_{2}}N^{(j_{1})}N^{(j_{2})}=(j_{1}+j_{2})\lambda^{-1}(\lambda^{-1}y)^{j_{1}+j_{2}}N^{(j_{1})}N^{(j_{2})}+&\\
+(\lambda^{-1}y)^{j_{1}+j_{2}+1}N^{(j_{1}+1)}N^{(j_{2})}+
(\lambda^{-1}y)^{j_{1}+j_{2}+1}N^{(j_{1})}N^{(j_{2}+1)}.&
\end{eqnarray*}
The terms in the right-hand side of the latter formula clearly belong to $J$ and so from \eqref{lm:ind2} we get that $L_{i+1}$ preserves $J$ as soon as 
$L_i$ does. The fact that $\lambda^{-1}yD_{y}$ preserves $B_{0}$ is proved in
exactly the same fashion, namely by first checking that
this is true for the generators of $B_0$ and then applying an inductive
argument based on the observation that
$\lambda^{-1}yD_{y}$ is a derivation.

Now that we established that both $B_0$ and $J$ are preserved by the operators $L_i$ we use this information to check the
induction step in formula (\ref{lm:form1}).
Since $\lambda^{-1}\in J$, equation (\ref{lm:ind2}) taken modulo $J$
gives 
\begin{equation}
L_{i+1}*v=(N-(i+1)\lambda^{-1}
)L_{i}*v-y\lambda^{-1}D_{y}(L_{i}*v)
=NL_{i}*v-y\lambda^{-1}D_{y}(L_{i}*v)
\label{lm:ind3}
\end{equation}
for  $v\in B_0$. Using the relations:\quad $y\lambda^{-1}D_{y}*N^{i+1}=(i+1)y
\lambda^{-1}N^{i}N^{(1)}$ \quad and
$$
y\lambda^{-1}D_{y}*N^{i-j}y^{j}\lambda^{-j}N^{(j)}=N^{i-j}y^{j+1}\lambda^{-(j+1)
}N^{(j+1)}
\; (\mathrm{ mod}\; J)
$$
we get from \eqref{lm:ind3} and the induction assumption that 
\begin{equation*}
\begin{split}
L_{i+1}*N=\,\,&
NL_{i}*N-y\lambda^{-1}D_{y}*(L_{i}*N)=\\
&N^{i+2}+\sum_{j=1}^{i}\binom{i+1}{j+1}(-1)^{j}N^{i+1-j}y^{j}\lambda^{-
j}N^{(j)}-(i+1)y\lambda^{-1}N^{i}N^{(1)}
-\\
&\sum_{j=1}^{i}\binom{i+1}{j+1}(-1)^{j}N^{i-j}y^{j+1}\lambda^{-(j+1)}N^{(j+1)}\;
(\mathrm{ mod}\; J).
\end{split}
\end{equation*}
          The usual properties of binomial coefficients now
accomplish the proof of the step of induction, which completes the proof of formula 
\eqref{lm:form1}.

The last statement in Lemma \ref{lem-6} 
follows trivially since the degree of the variable $\lambda^{-1}$ in
$\lambda^{-i}
y^{-i}\Theta (R_{i})$ (considered as an element of $B_{y,N}$) is
precisely $i-1$, see  (\ref{eq:k-1}), and the degree in $\la^{-1}$ of
the monomial  $(\la^{-1})^{\al_{1}}y^{\al_{2}}(\la^{-1}y)^{\vert I\vert} N^I$
  equals $\alpha_{1}+\vert I\vert$.
\end{proof}

\medskip\noindent
{\bf Description of the terms in equation (\ref{Dlog}).}
 Consider now the homogenized spectral pencil
         $T_{\lambda}=\sum_{i=0}^kQ_{i}(z)\lambda^{k-i}\frac{d^i}{dz^i}$.
             It acts (if we substitute $D^i_z$ for $\frac{d^i}{dz^i}$) on the differential module $A_{z,L}e^{M}$
	   and satisfies the obvious relation
            $$T_{\lambda}*e^{M}=\sum_{i=0}^kQ_{i}(z)\lambda^{k-i}R_{i}e^{M}.$$

Recall from Lemma \ref{firstformal} that each $R_i$ is a polynomial $R_i(\la,L^{(0)},\ldots)$. Now the non-linear equation 
	  $\sum_{i=0}^kQ_{i}(z)\la^{k-i}R_{i}=0$  in $L$  is the analog
over $A_{z,L}$
	  of the non-linear equation (\ref{Dlog}). A change of
variables
$y=z^{-1}$ and division by $\lambda^{k}$ transforms the previous equation into
the  equation
\begin{equation}
\label{eq:Dloginfty}
\lambda^{-k}S_{\lambda}(yN):=\Theta(\sum_{i=0}^kQ_{i}(z)\lambda^{-i}R_{i})=0.
\end{equation}

In what follows we  use the notation
\begin{multline*}
\Theta(Q_{i}(z))=\Theta(a_{i,i}z^i+a_{i,i-1}z^{i-1}+\ldots+a_{i,0})=\\
y^{-i}(a_{i,
i}+a_{i,i-1}y+\ldots+a_{i,0}y^i)=y^{-i}P_{i}(y),
\end{multline*}
             where $P_{i}(y)$ is a polynomial in $y$ of degree at most $i$.
             Equation \eqref{eq:Dloginfty} then takes the form
             	\begin{equation}
\label{eq:Dloginfty2}
\lambda^{-k}S_{\lambda}(yN):=\sum_{i=0}^kP_{i}(y)y^{-i}\lambda^{-i}\Theta(R_{i})
=0.
\end{equation}
Lemma~\ref{lem-6} leads to the following
statement.

\begin{Lemma}
	\label{transfeq}
In the above notation one has 
\begin{equation}
\lambda^{-k}S_{\lambda}(yN)= \sum_{i}P_i(y)N^{i}+
\sum_{i=0}^k\sum_{j=1}^{i-1}P_i(y)\binom{i}{j+1}(-1)^{j}N^{i-1-j}y^{j}
\lambda^{-j}N^{(j)} +b
\label{eq:start}
\end{equation}
for some $b\in J$.
        The exponents of all non-zero monomials
$(\lambda^{-1})^{\alpha_{1}}y^{\alpha_{2}}(\lambda^{-1}y)^{\vert
I\vert}N^I$ occurring in
the right-hand side
satisfy the condition $\alpha_{1}+\vert I\vert\leq k-1$.
\end{Lemma}

\medskip\noindent
{\bf  Generalities on power series}. Lemma~\ref{transfeq} tells us all we need to know about the explicit form of the terms in equation (\ref{Dlog}).
The {\em second step} in solving (\ref{Dlog}) is to find a 
recurrence relation for the formal power
series solutions of (\ref{eq:Dloginfty2}) using (\ref{eq:start}).
To do this we will use the Ansatz $L(y)=yN(y)=\sum_{i=1}^\infty \eps_{i}y^{i}$.
In algebraic terms this  means that
             instead of an element of $B_{y,N}$ we consider its image under
the map of
             differential rings $B_{y,N}\to {\mathbf
C}[\eps_{1},\eps_{2},\ldots][[y]]$
             given by $N\mapsto \sum_{i=0}^{\infty}\eps_{i+1}y^{i}$. (The latter
             ring is equipped with the usual derivation
            $\frac{d}{dy}$; in particular, $\frac{d}{dy}*\eps_{i}=0$.)

We will repeatedly use some easily verified
algebraic properties of the derivatives of formal power series similar to
$N(y)$. These properties are summed up below.

\begin{not+}
If $A=\sum_{i=0}^{\infty}k_{i}y^{i}\in K[[y]]$ is a formal power series
with coefficients in
a field $K$ we  define $[A]_{m}:=k_{m}.$  The notation ``$EXP\;\mathrm{ mod}(
\eps_{1},\ldots,\eps_{m})$'' means that expression $EXP$ is taken
modulo the vector space generated by
all monomials in the indicated variables
$(\eps_{1},\ldots,\eps_{m})$. (Note that this usage is different from
the earlier  used ``$(\text{mod } J)$'', where $J$ is some ideal.)
\end{not+}

\begin{Lemma} Let $N(y)=\sum_{i=0}^{\infty}\eps_{i+1}y^{i}\in {\mathbf
            C}[\eps_{1},\eps_{2},\ldots][[y]]$. Then the following relations
            hold:
           \label{powerseries}

            \begin{enumerate}
	 \item $[y^{i}N^{(i)}]_{m}=m(m-1)\ldots(m-i+1)\eps_{m+1}.$
	 If $0\leq m<i$, then $[y^iN^{(i)}]_{m}=0$.

	 \item One has
$$
[y^{\vert J\vert }N^{J}]_{m}=\sum_{c_{1}+\ldots+c_{s}=m}\prod_{i=1}^s
c_{i}(c_{i}-1)\ldots(c_{i}-j_{i}+1)\eps_{c_{i}+1}.
$$
(If $j_{i}=0$ then by abuse of notation we interpret
$c_{i}(c_{i}-1)\ldots (c_{i}-j_{i}+1)$ as $1$).
   If $J$ contains strictly more than one non-zero element then
   $[y^{\vert J\vert} N^{J}]_{m}= 0\; \mathrm{mod}
(\eps_{1},\ldots,\eps_{m})$.

\item
$[N^{j}]_{m}=j\eps_{1}^{j-1}\eps_{m+1}\; \mathrm{mod}
(\eps_{1},\ldots,\eps_{m})$ if $m\geq 1$ while $ [N^{j}]_{0}=\eps_{1}^{j}$.

\item
$[y^{i }N^{j}N^{(i)}]_{m}=m(m-1)\ldots(m-i+1)\eps_{1}^{j}\eps_{m+1}\;
\mathrm{mod}
(\eps_{1},\ldots,\eps_{m})$, if $i\geq 1$.

	 \item If the monomial
$M=(\lambda^{-1})^{\alpha_{1}}y^{\alpha_{2}}(\lambda^{-1}y)^{\vert
I\vert}N^I$ belongs to $B_{0}$ then
$$
[M]_{m}=B(\eps_{1},\ldots,\eps_{m+1},\lambda^{-1},\lambda^{-1}m),
$$
where $B(\eps_{1},\ldots,\eps_{m+1},x,y)$ is a polynomial of degree at most
$\vert
I\vert$ in $y$ and at most $\alpha_{1}$ in $x$.
\end{enumerate}
\end{Lemma}
\begin{proof} We prove properties  (2) and (5) leaving the rest as an
exercise for the interested  reader.  Note first that
        $$
        [y^{\vert J\vert }N^{J}]_{m}  =      \sum_{c_{1}+\ldots+c_{s}=m}\prod_{i=1}^s[y^{j_{i}}N^{j_{i}}]_{c_{i}}.
        $$
         By property (1), this gives the required relation in (2). 
Observe then that
the only case when one has
$$ \prod_{i=1}^s
c_{i}(c_{i}-1)\ldots(c_{i}-j_{i}+1)\eps_{c_{i}+1}\neq  0\; \mathrm{mod}
(\eps_{1},..\eps_{m}),
$$
is if some $\eps_{c_{i}+1}=\eps_{m+1}$, and then $c_{i}=m$, and
all other
$c_{l}=0,\ l\neq i$. But if for some such $l\neq i$ we have $j_{l}>0$, then
$c_{l}(c_{l}-1)\ldots(c_{l}-j_{l}+1)=0$; hence at most one $j_{i}\neq
0,\ 1\leq i\leq s$. This finishes the proof of (2).

In order to prove (5) notice that (1) implies that the expression
\begin{multline*}
A_{i}(\eps_{m+1},\la^{-1},\la^{-1}m):=[(\la^{-1}y)^iN^{(i)}]_{m}=\la^{-i}
m(m-1)\ldots(m-i+1)\eps_{m+1}=\\
(\la^{-1}m)(\la^{-1}m-\la^{-1})\ldots
(\la^{-1}m-(i-1)\la^{-1})\eps_{m+1}
\end{multline*}
is a polynomial 
whose degree in the last variable equals $i$. Therefore, the expression
$$A_{i}(\eps_{1},\ldots,\eps_{m+1},\la^{-1},\la^{-1}m):=[(\la^{-1}y)^{\vert J\vert} N^{J}]_{m}=\sum_{c_{1}+\ldots+c_{s}=m}
\prod_{i=1}^s[(\la^{-1}y)^{j_{i}}N^{j_{i}}]_{c_{i}}
        $$
is a polynomial 
whose degree in the last variable is at most $\vert J\vert$. This
implies (5).
\end{proof}

The following description of the coefficients of the power series in
the ideal $J$ is an immediate consequence of Lemma \ref{powerseries}.

\begin{Lemma}
           \label{powerseriesJ} Let
           $M=(\lambda^{-1})^{\alpha_{1}}y^{\alpha_{2}}(\lambda^{-1}y)^{\vert
I\vert}N^I$,
where $\alpha_{i}\in {\bZ_{\geq 0}},\ i=1,2$, $I\in FS$, be a
monomial in $B_{0}$. Then
\begin{enumerate}
         \item  $[M]_{0}\neq 0$ if and only if $M=(\lambda^{-1})^{\alpha_{1}}N^{j}$ for
           some $j$. If in addition $M\in J$ then $\alpha_1\geq 1$.

           \item If $M\in J$ and $m\ge 1$ then $[M]_{m}\neq 0 \; \mathrm{mod}
(\eps_{1},\ldots,\eps_{m})$ implies that:
\begin{itemize}
\item[(a)] $\alpha_{1}\geq 1$,
\item[(b)] $\alpha_{2}=0$,
\item[(c)] $I=(0,\ldots,0,i)$ with $lng(I)=j+1$ and contains at
most one non-zero element $i$.
\end{itemize}
If $i\ge 1$ then
$M=(\lambda^{-1})^{\alpha_{1}}(\lambda^{-1}y)^{i}N^{j}N^{(i)}$ and
$$
[M]_{m}=(\lambda^{-1})^{\alpha_{1}+i}m(m-1)\ldots(m-i+1)\eps_{1}^{j}\eps_{m+1}\;
\mathrm{mod}
(\eps_{1},\ldots,\eps_{m}).
$$

\item If $b\in J$ then
$[b]_{m}=\lambda^{-1}E(\eps_{1}, \lambda^{-1},\lambda^{-1}m)\eps_{m+1}\;
\mathrm{mod}
(\eps_{1},\ldots,\eps_{m})$, where $E$ is
a polynomial.
        \end{enumerate}
           \end{Lemma}

Now we have all the tools needed to get an adequate  picture of the
recurrence relation
corresponding to equation (\ref{Dlog}). We use  the Ansatz
$N(y)=\sum_{i=0}^{\infty}\eps_{i+1}y^{i}$ in equation
(\ref{eq:Dloginfty2}).

\medskip
\noindent
{\bf The constant term.}
The first step in solving our recurrence relation is to get an equation
for $\eps_{1}$.

\begin{Lemma} Let $N(y)=\sum_{i=0}^{\infty}\eps_{i+1}y^{i}$.
      The values of $\la$ for which $\lambda^{-k}S_{\lambda}(yN)$ has no
      constant term are precisely the solutions of
          $R(\eps_{1},\lambda)=0$. Here $R(\eps_{1},\lambda)$ is given
          by
           \label{recdeg0}
           \begin{equation}
	R(\eps_{1},\lambda):=[\lambda^{-k}S_{\lambda}(yN)]_{0}=\sum_{i=0}^{k}a_{ii}\eps_{1}^{i}+
	\lambda^{-1}E(\eps_{1},\lambda^{-1})
	\label{eq:recdeg0}
	\end{equation}
	for some $E(\eps_{1},\lambda^{-1})$ which is a polynomial of degree at
	most $k$ in $\eps_{1}$ (and it has degree at most $k-2$ in $\lambda^{-1}$).
           \end{Lemma}
\begin{proof}

Part 1 of Lemma~\ref{powerseriesJ} and part 3 of Lemma~\ref{powerseries}
imply that
$[P_i(y)N^{i}]_{0}=a_{ii}\eps_{1}^{i}$ and
$[P_i(y)N^{i-1-j}y^{j}\lambda^{-j}N^{(
j)}]_{0}=0$ for all $j\ge 1$.
Hence by Lemma \ref{transfeq} 
the constant term of $\lambda^{-k}S_{\lambda}(yN)$ is
given by
$$
\sum_{i}a_{ii}\eps_{1}^{i} +[b]_{0}.
$$
By  Lemma \ref{powerseriesJ} part 3, there is a polynomial $E$ such that $[b]_0=\lambda^{-1}E(\eps_{1},\lambda^{-1})$. Part 1 of this same lemma gives that the only
terms in $b$
 contributing a non-zero term of degree $j$ in
$\eps_{1}$ to $[b]_0$  come from terms of the form
$c(\lambda^{-1})^{\alpha_{1}}N^{j}$ with $\alpha_{1}\geq 1$ and
$c\in {\bC}$. It is clear that such
a term will have $j\leq k$  since by (\ref{eq:Dloginfty})
and (\ref{eq:theta}) it stems from some
$\lambda^{-i} y^{-i}\Theta (R_{i})=y^{-i}(yN-y^2\lambda^{-1}D_{y})^{i} *1$ 
with $i\leq k$,  which contains $N$  at most in the power $i$. By the same observation it follows that the degree in $\lambda^{-1}$ of such a term is less than or equal to $k-1$.
\end{proof}

Hence  the initial step in  computing
the formal solution to (\ref{eq:Dloginfty}) is solving the  equation
\begin{equation}
	R(\eps_{1},\lambda)=\sum_{i=0}^{k}a_{ii}\eps_1^{i}+
	\lambda^{-1}E(\eps_{1},\lambda^{-1})=0.
	\label{eq:for a1}
	\end{equation}
Note that this equation tends
             to the equation 
             $
\sum_{i=0}^k a_{ii}\eps_{1}^{i}=0
$
as $\lambda\to  \infty$ 
and that the latter coincides with equation (\ref{lambdadef}).
In particular, under the assumptions of Theorem \ref{th:main}, or
equivalently, of  Lemma \ref{lm:gen}, equation (\ref{eq:for a1})
has $k$ distinct solutions for any sufficiently large value of $\lambda$.
Choose one of the branches $\eps_{1}=\eps_{1}(\lambda)$ that  solves
(\ref{eq:for a1}) in a neighborhood of $\lambda=\infty$ in $\bCP$.
This means that we have determined $\eps_{1}$  for any large enough value
of $\lambda$ and we can start determining the other coefficients of
$N(y)=\sum_{i=0}^{\infty}\eps_{i+1}y^{i}$ in terms of the chosen $\eps_{1}$.

\medskip
\noindent
{\bf The recurrence formula.}
Assume now 
that the coefficients
$\eps_{1},\ldots,\eps_{m}$ of $N$ have already been defined. The recursive
step that defines $\eps_{m+1}$ then corresponds to solving the equation
$[\lambda^{-k}S_{\lambda}(yN)]_{m}=0$.

\begin{Lemma}
           \label{rec.rel} One has
$$[\lambda^{-k}S_{\lambda}(yN(y))]_{m}=\Phi_{0}(\eps_{1},\lambda^{-1},\lambda^{-1
}m)\eps_{m+1}+
B(\eps_{1},\ldots,\eps_{m},\lambda^{-1},\lambda^{-1}m),$$
where $B$  and $\Phi_{0}$ are polynomials and the latter satisfies 
\begin{equation}
\Phi_{0}(\eps_{1},\lambda^{-1},\lambda^{-1}m)=-\sum_{i=0}^ka_{ii}\frac{(\eps_{1}
-\lambda^{-1}m)^{i}-\eps_{1}^{i}}{\lambda^{-1}m}+
           \lambda^{-1}E(\eps_{1},\lambda^{-1},\lambda^{-1}m),
           \label{eq:Phi}
\end{equation}
where $E$  is a polynomial whose degree in the third variable (which is
$\lambda^{-1}m$) is strictly
less than $k-2$. The
degree of $B$ in the last variable (which is $\lambda^{-1}m$ as well)
is less than or
equal to $k-1$.
          \end{Lemma}

\begin{proof} By equation (\ref{eq:start}) one has 
           \begin{multline}
[\lambda^{-k}S_{\lambda}(yN)]_{m}=
\sum_{i}[P_iN^{i}]_{m}+\\
\sum_{i=0}^k\sum_{j=1}^{i-1}\binom{i}{j+1}(-1)^{j}
[P_iN^{i-1-j}y^{j}\lambda^{-j}N^{(j)}]_{m} +[b]_{m}.
\label{eq:10}
\end{multline}
By Lemma \ref{powerseries} the following two relations hold
$\mathrm{mod}(\eps_{1},\ldots,\eps_{m})$:
$$
[P_iN^{i}]_{m}=a_{ii}i\eps_{1}^{i-1}\eps_{m+1}=a_{ii}\binom{i}{1}\eps_{1}^{i-1}\eps_{m+1}
$$
and
$$
[P_iN^{i-1-j}y^{j}\lambda^{-j}N^{(j)}]_{m}=a_{ii}\lambda^{-j}m(m-1)\ldots(m-j+1)
\eps_{1}^{i-1-j}\eps_{m+1}.
$$
Finally, by Lemma \ref{powerseriesJ} one has
$[b]_{m}=\lambda^{-1}E_{2}(\eps_{1},\lambda^{-1},\lambda^{-1}m)\eps_{m+1}$,
where $E_{2}$ is a
polynomial.
The identity
$\lambda^{-j}m(m-1)\ldots(m-j+1)=\lambda^{-j}m^{j}+\lambda^{-1}F_{i}(\lambda^{-1
},\lambda^{-1}m)$, where
$F_{i}$ is a
polynomial,  implies that $\mathrm{mod}(\eps_{1},\ldots,\eps_{m})$ one further has
\begin{multline*}
[\lambda^{-k}S_{\lambda}(yN)]_{m}=
(\sum_{i=0}^ka_{ii}\sum_{j=0}^{i-1}\binom{i}{j+1}(-1)^{j}\eps_{1}^{i-1-j}
\lambda^{-j}m^{j}+\\
\lambda^{-1}E(\eps_{1},\lambda^{-1},\lambda^{-1}m))\eps_{m+1}.
\end{multline*}
The sum over $j$ in the latter expression  is $-\lambda
m^{-1}((\eps_{1}-\lambda^{-1}m)^{i}-\eps_{1}^{i})$. This completes the
proof of the formula for $\Phi_{0}$. The bounds for the degrees of the
involved polynomials stated in the lemma 
follow directly from the last statement of Lemma \ref{transfeq} and  part (5)
of Lemma \ref{powerseries}.
\end{proof}

\begin{rema+}
Once we prove -- which we will do in the next section -- that there is an open set $D$ containing $z=\infty$ where 
$\Phi_{0}(\eps_{1},\lambda^{-1},\lambda^{-1}m)\neq 0$ for all $m$ and $\lambda\in \Omega$ then by Lemma~\ref{rec.rel} we immediately get a recursive determination of $\eps_{m+1}(\lambda)$. Indeed, the initial data is the choice of $\eps_1(\lambda)$. In (\ref{eq:for a1}) we saw that $\lim_{\lambda \to\infty}\eps_1(\lambda)$ is a root of (\ref{lambdadef}) so in particular it is bounded. Since in the right-hand side of (\ref{eq:10}) the only part not being a multiple of  $\lambda^{-1}$ is $\sum_{i}[P_iN^{i}]_{m}$, we get that $yN_\lambda(y)=\sum_{i=1}^\infty \eps_{i}(\lambda)y^{i}$ converges formally to a formal power series $yN(y)=\sum_{i=1}^\infty \eps_{i}y^{i}$ which is a formal solution to the equation $\sum_{i}P_iN^{i}=0$. This shows that $yN(y)$ is a formal solution (at infinity) to the equation of the plane curve associated with the pencil $T_\lambda$.
\label{Rmk:formal}

\end{rema+}
\medskip\noindent
{\bf Summary.} We have considered the Ansatz $$L(z)=yN(y)=\sum_{i=1}^\infty \eps_{i}y^{i}$$
for equation (\ref{Dlog}) at $z=\infty$ in the form given by the change of variable $y=1/z$ (cf.~(\ref{eq:Dloginfty})). We saw by Lemma \ref{recdeg0} that $\eps_{1}=\eps_{1}(\lambda) $ has to be a solution to equation $(\ref{eq:for a1})$ and that it is always possible to find such a solution for large $\lambda$ under the assumption of Theorem \ref{th:main}. Furthermore, 
Lemma \ref{rec.rel} is immediately reformulated into the recurrence relation 
\begin{equation}
\eps_{m+1}=-\Phi_{0}(\eps_{1},\lambda^{-1},\lambda^{-1}m)^{-1}B(\eps_{1},\ldots,
\eps_{m},\lambda^{-1},\lambda^{-1}m),
           \label{eq:rec.rel}
\end{equation}
which  formally solves equation (\ref{eq:Dloginfty}) under the
assumption that  for all $m=1,2,\ldots$ one has
$\Phi_{0}(\eps_{1},\lambda^{-1},\lambda^{-1}m)\neq
0$. Clearly, the next step is to study this last condition. This will be done in the next section.

\begin{Ex}
Consider
$T_{\la}=\sum_{i=0}^2Q_{i}(z)\la^{2-i}\frac{d^i}{dz^i}$ with $Q_{i}(z)=\sum_{j=0}^{i}a_{ij}z^j$, $i=0,1,2$.
Then
$$\Phi_{0}(\eps_{1},\lambda^{-1},\lambda^{-1}m)=a_{22}(2\eps_{1}-\lambda^{-1}m)+
a_{11} -a_{22}\lambda^{-1}
$$
with $E=-a_{22}$.

The tricky point in proving that there actually {\em exists} a formal
solution to (\ref{Dlog})
is that the polynomials $\Phi_{0}(\eps_{1},\lambda^{-1},\lambda^{-1}m)$
might in fact vanish for some values of $m$. Moreover, proving the
{\em convergence} of formal
solutions is also rendered difficult by the possibility that
$\Phi_{0}(\eps_{1},\la^{-1},\la^{-1}m)\to
0 $ when $m\to \infty$.  We can see from Lemma~\ref{rec.rel} and the
   above example that the
terms in $\Phi_{0}(\eps_{1},\lambda^{-1},\lambda^{-1}m)$ have  degree 
in $m$ that is less than or equal
to their total degree in $\lambda^{-1}$.
This observation will be the crucial point in the proof
   of the convergence of formal power series solutions given in
the next section.
\end{Ex}

\section{Estimating the radius of convergence}

We use the classical
method of majorants (see e.g.~\cite{HL} or \cite{Hu}) to show  that under
certain conditions the formal power series solutions to equation 
(\ref{eq:Dloginfty}) actually represent analytic functions.
	The idea is to find an upper bound for the solutions to the
studied  recurrence relation
	 in terms of  a simpler and explicitly
	solvable one. In other words, we want to substitute the original
recurrence relation  by a
	simpler recurrence
ensuring that during this
	process the absolute value of the solutions will not  decrease.
	The recurrence relation (\ref{eq:rec.rel}) has,
	for fixed $\lambda$, the form
	\begin{equation}
	   \eps_{m+1}=G_{m}(\eps_{1},\ldots,\eps_{m}),
\label{eq:gen.rec.rel}
	\end{equation}
	where
	$G_{m}(\eps_{1},\ldots,\eps_{m})=\sum_{I}\de^{m}_{I}\eps^{I}$, 
$m=1,2,\ldots$, 
	is a family of polynomials.
\medskip
The following lemma is borrowed from \cite{HL}.

	\begin{Lemma}
		Given a recurrence relation
(\ref{eq:gen.rec.rel}) assume that
		$$
		H_{m}(\eps_{1},\ldots,\eps_{m})	=	\sum^{m}_{I}A_{I}\eps^{I},\quad m=1,2,\ldots,
		$$
	is a family of polynomials with positive coefficients satisfying
	\begin{equation}
	    \vert G_{m}(\eps_{1},\ldots,\eps_{m})\vert\leq H_{m}(\vert
\eps_{1}\vert
	,\ldots,\vert \eps_{m}\vert )
	 \label{eq:HL}
	\end{equation}
	for all $\eps_{i}\in \bC$.
(This is true, for example, if $A^{m}_{I}\geq \vert \de^{m}_{I}\vert )$.
Take
	$E_{1}\geq \vert \eps_{1}\vert$ and define inductively
	$E_{m+1}= H_{m}(E_{1},\ldots,E_{m})$.
	Then for any $m$ one has that $E_{m}\geq \vert \eps_{m}\vert$ and
hence
	the radius of convergence of the series $\sum_{i=1}^\infty
\eps_{i}t^i$ is greater than or equal to 
	the
	radius of convergence of the series $\sum_{i=1}^\infty
	E_{i}t^i$.
	\label{lm:HL}
	\end{Lemma}

         Lemma \ref{lm:HL} implies the following statement.

\begin{Prop}
           \label{lm:ind}
           Consider a set $\Lambda\subset \bCP$
containing $\infty$ and such that $\lambda=0$ is an interior
point in  $\bCP\setminus \Lambda$. Let
$\eps_{1}=\eps_{1}(\lambda)$, $\lambda\in \Lambda$, be
a continuous solution to (\ref{eq:for a1}). If
            the polynomials
$\Phi_{0}(\eps_{1}(\lambda),\lambda^{-1},\lambda^{-1}m)$, $m=1,2,\ldots$, 
satisfy the condition 
	\begin{equation}
	 \sup
\{\vert\Phi_{0}(\eps_{1}(\lambda),\lambda^{-1},\lambda^{-1}m)
^{-1}(\lambda^{-1}m)^{r}\vert \,:\,\lambda\in \Lambda,\ m\geq 0,\ r=0,1,\ldots,
k-1 \}<\infty
	    \label{eq:conv1}
	\end{equation}
then there exists $D>0$ such that the series
$\sum_{i\geq 1}\eps_{i}y^i$ defined by the recurrence
(\ref{eq:rec.rel}) converges for any $ \lambda\in\Lambda$
in the disk $\vert y\vert < D$.
\label{pr:conv}
	\end{Prop}

	\begin{proof}
By Lemma \ref{transfeq} we know that
\begin{equation}
           \lambda^{-k}S_{\lambda}(yN)=\sum c_{\alpha_{1} \alpha_{2} J}
(\lambda^{-1})^{\alpha_{1}}y^{\alpha_{2}}(\lambda^{-1}y)^{\vert
J\vert}N^J,
\label{exp:eq}
\end{equation}
where the sum is taken over a finite set of indices $\alpha_{i}\in
{\bZ}_{\geq 0}$ and $J\in FS$ such that
$\alpha_{1}+\vert J\vert\leq k-1$.

        The choice of $\eps_{1}$ (depending
on $\lambda$) 
determines $N(y)=\eps_{1}+M(y)$, where $M(y)=\sum_{i=1}^\infty
\eps_{i+1}y^{i}$ has no constant term.
If $J$ starts with precisely $l$ zeros, i.e., $J=0(l)\cup J_{+}$,
then
$$N^{J}=(\eps_{1}+M)^lM^{J_{+}}=\sum_{s=0}^{l}
\binom{l}{s}\eps_{1}^{l-s}M^{0(s)\cup J_{+}}=\sum_{J_{1}\in W}
d_{J_{1}}M^{J_{1}},
$$
where all $J_{1}$'s that occur in the last sum  satisfy the relation $\vert
J_{1}\vert =\vert J\vert $.
Using this and substituting $M(y)$ into (\ref{exp:eq}), we get the equation
\begin{equation}
	    \label{exp:eq2}
        \sum d_{\alpha_{1} \alpha_{2} I}
(\lambda^{-1})^{\alpha_{1}}y^{\alpha_{2}}(\lambda^{-1}y)^{\vert
I\vert}M^I=0,
\end{equation}
where the relation $\alpha_{1}+\vert I\vert\leq k-1$ is still valid.
Note that  (\ref{exp:eq2}) contains no non-zero $d_{\al_{1}\al_{2}J}$
such that $\al_{2}=0$ and $\vert J\vert=0$ since this case will
contribute to a nontrivial constant term (degree $0$ in $y$). But the
constant term was already eliminated by the explicit choice of $\eps_{1}$.
The coefficients in (\ref{exp:eq2})
are polynomials in $\eps_{1}(\lambda)$. Since by the discussion of the asymptotic behavior of equation (\ref{eq:for a1}) for $\eps_{1}(\lambda)$ we know that 
$\lim_{\lambda\to\infty}\eps_{1}(\lambda)$
is finite, these coefficients are bounded
in the domain $\Lambda$.


Now we will analyze how the recurrence formula (\ref{eq:rec.rel}) is
affected by our choice of a branch $\eps_{1}(\la)$. We will do this in
a way similar to the arguments in Lemma~\ref{rec.rel}.
Since the constant term of $M$ vanishes by Lemma
\ref{powerseries} parts (2)-(4) one gets that only the terms
$(\lambda^{-1})^{\alpha_{1}}y^{\alpha_{2}}(\lambda^{-1}y)^{\vert
I\vert}M^I$ 
for which
             $I=(l)$ has length $1$ and $\alpha_{2}=0$ will satisfy the
             condition that
             the $m$-th coefficient is non-zero
             $\mathrm{mod}(\eps_{1},\ldots,\eps_{m})$. For these terms one has
             $$[(\lambda^{-1})^{\alpha_{1}}(\lambda^{-1}y)^{l}M^{(l)}]_{m}    =         (\lambda^{-1})^{\alpha_{1}}m(m-1)\ldots(m-l+1)\eps_{m+1}.$$
             Hence we can split the terms in the left-hand side of
             equation (\ref{exp:eq2}) into two groups, namely,  the sum
             $
             \sum_{\alpha_{1},l} d_{\alpha_{1} 0
(l)}(\lambda^{-1})^{\alpha_{1}}(\lambda^{-1}y)^{l}M^{(l)}
             $
             and the remaining terms
$ \sum _{V} d_{\alpha_{1} \alpha_{2} I}
(\lambda^{-1})^{\alpha_{1}}y^{\alpha_{2}}(\lambda^{-1}y)^{\vert
I\vert}M^I$. These remaining terms
correspond to a certain index set $V$. The first sum has the
$m$-th coefficient
equal to $\Phi_{0}(\eps_{1},\lambda^{-1},\lambda^{-1}m)\eps_{m+1}$,
while the $m$-th coefficient in the second sum is a polynomial in the
preceding coefficients. By the hypothesis of Proposition~\ref{pr:conv} we have 
$\Phi_{0}(\eps_{1},\lambda^{-1},\lambda^{-1}m)\neq 0$ if $\lambda\in
\Lambda$.
Hence the recurrence relation expressing $\eps_{m+1}$ 
        can be rewritten as
             \begin{equation}
\eps_{m+1}=-\Phi_{0}(\eps_{1},\lambda^{-1},\lambda^{-1}m)^{-1}\sum_{V}
d_{\alpha_{1} \alpha_{2} I}
[(\lambda^{-1})^{\alpha_{1}}y^{\alpha_{2}}(\lambda^{-1}y)^{\vert
I\vert}M^I]_{m}.
	\label{eq:konvradie}
\end{equation}
Interpret its right-hand side as a polynomial
$G_{m}(\eps_{2},\ldots,\eps_{m})$ with coefficients depending on
$\lambda$.
If $K= max \{ \vert d_{\alpha_{1}
\alpha_{2} I}\vert,\ (\alpha_{1},
\alpha_{2}, I)\in V\}$ then
clearly
\begin{equation}
           \label{eq:rec1}
           \vert G_{m}(\eps_{2},\ldots,\eps_{m})\vert \leq K\sum_{V}\vert
\Phi_{0}(\eps_{1},\lambda^{-1},\lambda^{-1}m)^{-1}[(\lambda^{-1})^{\alpha_{1}}y^{\alpha_{2}}(\lambda^{-1}y)^{\vert
I\vert}M^I]_{m}\vert.
\end{equation}

We can find an even simpler upper bound for these terms.
Set $J=(j_{1},j_{2},..,j_{s})$ and recall that
$$
[y^{\vert J\vert }M^{J}]_{m}=\sum_{c_{1}+\ldots+c_{s}=m}\prod_{i=1}^s
c_{i}(c_{i}-1)\ldots
(c_{i}-j_{i}+1)\eps_{c_{i}}
$$
by Lemma \ref{powerseries}.
Clearly,  $\prod_{i=1}^s c_{i}(c_{i}-1)\ldots
(c_{i}-j_{i}+1)\leq m^{\vert
J\vert}$ and thus we get
$$
\vert[y^{\vert J\vert}M^{J}]_{m}\vert \leq m^{\vert J\vert}
\sum_{c_{1}+\ldots+c_{s}=m}\prod_{i=1}^s
\vert \eps_{c_{i}}\vert .
$$

Setting $ M_{abs}:=\sum_{i=2}^{\infty}\vert
\eps_{i}\vert t^i$ and noting that $s=lng(J)$ we obtain
$$
\sum_{c_{1}+\ldots+c_{s}=m}\prod_{i=1}^s
\vert \eps_{c_{i}}\vert =[M_{abs}
^{lng(J)}]_{m}
$$
and so we have the bound 
$$
\vert[y^{\vert J\vert}M^{J}]_{m}\vert \leq m^{\vert J\vert}[M_{abs}
^{lng(J)}]_{m}.
$$

This implies the inequality
\begin{equation}
\vert G_{m}(\eps_{2},\ldots,\eps_{m})\vert \leq K\sum_{V}\vert
\Phi_{0}(\eps_{1},\lambda^{-1},\lambda^{-1}m)^{-1}\lambda^{-\alpha_{1}}(\lambda^
{-1}m)^{\vert
I\vert}\vert [y^{\alpha_{2}}M_{abs}
^{lng(I)}]_{m}.
\label{eq:for}
    \end{equation}

As it was mentioned in the beginning of the proof, for each term in
this sum one has $\al_{1}+\vert I\vert \le k-1$.
Hence both inequalities $\al_{1}\le k-1$ and $\vert I\vert \le k-1$ hold.
The assumption that $0$ is an interior point of the complement of
$\Lambda$ implies that there is a positive real number $D$ such that
$\vert \lambda^{-i}\vert\leq D$ for $\lambda\in \Lambda$ and
$i=1,\ldots,k-1$. By the second assumption of the proposition there is a
constant $C$ such that
$$
\vert
\Phi_{0}(\eps_{1},\lambda^{-1},\lambda^{-1}m)^{-1}(\lambda^{-1}m)^{\vert
I\vert}\vert\leq C
$$
for $\lambda\in \Lambda$ and any $I$ occurring in (\ref{eq:for}).
So
\begin{multline*}
\vert G_{m}(\eps_{2},\ldots,\eps_{m})\vert
\leq KCD\sum_{V} [y^{\alpha_{2}}M_{abs}
^{lng(I)}]_{m}\leq \\
L\sum_{i,j}
[y^iM_{abs}^j]_{m}=:H_{m}(\vert \eps_{2}\vert,\ldots,\vert \eps_{m}\vert)
\end{multline*}
for some positive real number $L:=KCDr$,
where $r$ is the maximal number of indices in $V$ (the index set introduced above) 
that correspond to some $i=\alpha_{2}$ and $j=lng(I)$. Note that
$L$ is independent of $m$, since the index set $V$ just depends on $k$. The summation in $H_{m}$
is taken over
all pairs of non-negative integers $(i,j)$ except $(0,0)$ and $(0,1)$.
(The fact that $(0,0)$ is not needed follows from the observation that there
are no terms in
(\ref{exp:eq2}) with $\alpha_2=0$
and $lng(I)=0$; the fact that $(0,1)$ is not needed 
is implied by the observation that the terms with $\alpha_{2}=0$ and
$lng(I)=1$
were excluded from $V$ by the splitting of terms in (\ref{exp:eq2})  for the construction of $V$.) This is an infinite sum, but
we still get a well-defined polynomial since
$M$ contains no constant term.

The recurrence formulae $\eps_{m+1}=H_{m}(\eps_{2},\ldots,\eps_{m})$, 
$m\geq 1$, 
correspond to equating the $m$-th coefficients in both sides of the
relation 
$$
M= L\sum_{i,j} y^iM^j= L\left(\frac{1}{(1-y)(1-M)}-1-M\right) 
\label{eq:nth}
$$
defining the power series $M(y)=\sum_{i=1}^\infty\eps_{i+1}y^{i}$.
(Note that for $m=1$ we simply get  $\eps_{2}=H_{1}\in
{\bR}.$)
The above equation
is equivalent to the quadratic equation
$$(1+L)M^2-M+Ly/(1-y)=AM^2-M+C=0$$
and its solutions  are
$\frac{1}{2A}\pm{\frac{\sqrt{1-4CA}}{2A}}.$
These solutions are univalent analytic functions in case
$\vert 4CA\vert<1$. The latter condition is easily seen to be satisfied for
$\vert y\vert <\frac {1}{(1+2L)^2}$. These solutions
are developable in power series in this
disk and their power series will satisfy our recurrence formulae. This  
      accomplishes  the proof of Proposition \ref{pr:conv}.
Indeed, we have found a
sequence of  polynomials
that satisfy the assumptions of Lemma~\ref{lm:HL} and proved that the
formal power series that they define is
analytically convergent in some disk around $y=0$.
(This argument mimics the elegant proof that an algebraic
function has a convergent power series expansion given in
\cite[pp.~34-35]{HL}.)
\end{proof}

\begin{rema+} It is worth mentioning  that if  we know the supremum in
Proposition \ref{lm:ind} we can easily get  
the estimate $\frac {1}{(1+2L)^2}$ for the radius of convergence
of the solutions to equation (\ref{eq:start}) at least for small $k$. 
\end{rema+}

\medskip \noindent
{\bf Convergence of polynomial solutions.}
We will now show that the condition on the supremum in Proposition \ref{pr:conv} is generically satisfied for the 
families of polynomial eigenfunctions used in Propositions~\ref{lm:basic} and ~\ref{dlog}. This means that we assume that the corresponding pencil is of general type. If
$p_{n}(z)$ is any monic degree $n$ polynomial then its normalized  
logarithmic derivative $L_{n}(z)=\frac{p'_{n}(z)}{\lambda_{n} p_{n}(z)}$
expanded at infinity using the variable $y=1/z$ satisfies
the relation
$$
L_{n}(z)=yN(y)=\frac{n}{\lambda_{n}}y+\ldots =\sum_{i=1}^\infty \eps_{i,n}y^{i}.
$$
In particular, $\eps_{1,n}=\frac {n}{\lambda_{n}}.$ (Note that so far $\la_{n}\neq
0$ is  any constant.)

Now fix $\la_n$ to be one of the families of eigenvalues in Propositions~\ref{lm:basic}. In a certain neighborhood of infinity $\Lambda$ there is then for $n$ so large that $\la_n\in \Lambda$ a unique degree $n$ eigenpolynomial $p_{n}(z)$ that satisfies
$T_{\la_{n}}p_{n}(z)=0$. By Proposition~\ref{lm:basic}
for $n$ large enough one has 
$
    \lambda_{n}\sim \al n,
            $
             where $\al$ is one of the $k$ distinct roots of
             (\ref{eq:char}).
This gives $$ \lambda_{n}^{-1}n=\eps_{1,n}\to 1/\al$$
 as $n\to\infty$. Note that $1/\al$ is a root of the
reciprocal characteristic equation (\ref{lambdadef}):
\begin{equation}
	\label{lambdadef2}
	a_{k,k}\xi^{k}+a_{k-1,k-1}\xi^{k-1}+\ldots+a_{0,0}=0.
	\end{equation}
	
We next show that the assumption in Proposition \ref{lm:ind}
ensuring the convergence of our formal power series solutions is
satisfied for a large class of cases.

\begin{Prop}
	\label{prop:P_{0}}
	Assume that
\begin{itemize}
\item[(i)] the polynomial 
$\sum_{i=0}^{k}a_{ii}t^{i}$ has degree $k$ (i.e., $a_{kk}\neq 0$)
	\item[(ii)]  the root $\alpha^{-1}$ of
 $\sum_{i=0}^{k}a_{ii}t^{i}=0$ is simple,
\item[(iii)] for any $r>0$ the number  $(1-r)\alpha^{-1}$
is never a root of this same polynomial (in particular, taking $r=1$
implies that $0$  is not a root of the above polynomial, which gives  $a_{00}\neq
0$).
\end{itemize}
Then there exists a neighborhood $\widetilde\Lambda\in \bC$ of infinity with the following property: if  $n\in \bZ_{\geq 0}$ is such that $\la_n\in \widetilde\Lambda$  and $\eps_{1,n}=\frac {n}{\lambda_{n}}, $  then
\begin{equation*}
	 \sup
\{\vert\Phi_{0}(\eps_{1,n},\lambda_n^{-1},\lambda_n^{-1}m)
^{-1}(\lambda_n^{-1}m)^{r}\vert \,:\,m\geq 0,\ r=0,1,\ldots,
k-1 \}<\infty
	\end{equation*} 
and so the condition in Proposition \ref{lm:ind} is
valid.
\end{Prop}

	\begin{proof}Recall first that by Lemma \ref{rec.rel} we have
	 \begin{equation}
	 \label{eq:raa}
           \Phi_{0}(\eps_{1},\lambda^{-1},\lambda^{-1}m)=-\sum_{l=0}^ka_{ll}
           \frac{(\eps_{1}-\lambda^{-1}m)^{l}-\eps_{1}^{l}}{\lambda^{-1}m}+
           \lambda^{-1}E(\eps_{1},\lambda^{-1},\lambda^{-1}m).
           \end{equation}

	 We need to establish the existence of a finite supremum in the
	 previous proposition. Argue by contradiction: if the proposition is not true there are sequences of values 
	    $m_j$ and $n_j$, $j=1,2,\ldots$, such that 
	    \begin{itemize}
	    \item $\lim_{j\to\infty}n_j=\infty$,
	    \item  for some $i\leq k-1$ one has 
$(\lambda_{n_j}^{-1}m_j)^{-i}\Phi_{0}(\eps_{1,{n_j}},\lambda_{n_j}^{-1},\lambda_{n_j}^{-1}m_j)\to 0$
as 
	    $j\to \infty$.
	    \end{itemize}
	    We will in the following use the notation $n=n_j$, $m=m_j$, $\eps_{1}=\eps_{1,n}=\frac {n}{\lambda_{n}},$ and $\la=\la_n\in \Lambda$, where $n\in \bZ_{\geq 0}$, and thus supress indices in the sequence. All limits will be taken as $j\to \infty$.
	 
 By taking a subsequence we see that there exist two
possibilities: either there
	    is a sequence as above for which $m/n$ converges to a finite
	    limit $r\in \mathbb{R}$, or there is such a sequence for which
	    $m/n\to \infty$.  
	    In the second case we first note that also $\lambda^{-1}m\to \infty$ (since $\lim \lambda^{-1}m=\lim\lambda^{-1}n\cdot\lim(m/n)= \al^{-1}\lim m/n$). Hence by our asssumption also
	    $(\la^{-1}m)^{-(k-1)}\Phi_{0}(\eps_{1},\la^{-1},\la^{-1}m)\to 0$.
	    To get a contradiction we use the fact that
	    the degree in $\lambda^{-1}m$ of $E$ is less than $k-2$  by Lemma
	    \ref{rec.rel}, and that
	    $\lambda^{-1}\to 0$ for the chosen sequence (recall that $\lambda$ 
depends on $n_j$ and $n_j\to\infty$).
	    It follows that $\lim(\lambda^{-1}m)^{-(k-1)}
	    E(\eps_{1},\lambda^{-1},\lambda^{-1}m)=0$. Then 
	    the hypothesis that $a_{kk}\neq 0$ and the explicit description in (\ref{eq:raa}) imply that
	    $\lim(\lambda^{-1}m)^{-(k-1)}
	    \Phi_{0}(\eps_{1},\lambda^{-1},\lambda^{-1}m)=(-1)^{k-1}a_{kk}\neq 0$, which is the desired 
        contradiction.

	    In the first case, the
	    hypothesis implies that (for our sequence of $m$ and $n$) $\lambda^{-1}m=(\lambda^{-1}n)(m/n)\to r/\alpha$.  Furthermore	   
	    \begin{equation}
	    \label{eq:extralim}
	    \lim
	    \Phi_{0}(\eps_{1},\lambda^{-1},\lambda^{-1}m)=(r/\alpha)^{i}\lim((\la^{-1}m)^{-i}\Phi_{0}(\eps_{1},\lambda^{-1},\lambda^{-1}m))=0.
	    \end{equation}

Now we also have $\lambda\to \infty$, which implies that $\lim
	    \lambda^{-1}E(\eps_{1},\lambda^{-1},\lambda^{-1}m)=0$ and
thus (\ref{eq:extralim}) yields 
	     \begin{equation}
	     \label{eq:finall}
	   \lim
\sum_{l=0}^k a_{ll}\frac{(\eps_{1}-\lambda^{-1}m)^{l}-\eps_{1}^{l}}{\lambda^{-1}m
}	   =0.
	    \end{equation}

	    If $r\neq 0$ this immediately implies that both $\alpha^{-1}-r\alpha^{-1}$
and $\alpha^{-1}=\lim\eps_{1}$ solve the same equation, namely  
$\sum_{l=0}^{k}a_{ll}t^{l}=0$. This
	    gives a contradiction to the assumed conditions on this equation. On the other hand, if $r=0$ then a calculation shows that 
	the left-hand side of  (\ref{eq:finall}) is $\frac{d}{dt}(\sum_{l=0}^{k}a_{ll}t^{l})\big|_{t=\alpha^{-1}}$, which contradicts the fact that $\alpha^{-1}$ is a simple root and finishes the
	    proof. 
\end{proof}

	    \section {Proof of Theorems \ref{th:main}  and \ref{th:c}}
	    \label{sec:5}
	     We start with Theorem~\ref{th:c}.  Consider a pencil $T_\lambda$ of general type as defined in the introduction. By Proposition \ref{lm:basic} its polynomial eigenfunctions split into $k$ distinct families $p_{n,j}$, $j=1,\ldots,k$. Fix one of these families $p_n:=p_{n,j}$ with corresponding eigenvalues $\la_n\sim \alpha n$ as $n\to\infty$.

	    For each $\la$ we have an equation (\ref{Dlog}). Among its solutions in an open set $\Omega$ are all normalized logarithmic derivatives $\frac{p'}{\la p}$ of eigenfunctions $p$ (i.e., $T_\lambda*p=0$) that are defined and non-vanishing in $\Omega$. 
	 The recurrence relation (\ref{eq:rec.rel}) -- that we studied in the two last sections -- solves equation  (\ref{Dlog}) by means of the power series expansion at $y^{-1}=z=\infty$ of a solution of the special form $yN(y)=\sum_{i=0}^{\infty}\eps_{i+1}y^{i+1}$ (see Lemma \ref{rec.rel} in \S \ref{sec:3}). In Lemma  \ref{recdeg0} it was shown that for each $\la$ large enough there are $k$ possible initial values of $\eps_1=\eps_1(\la)$.  Let $\la=\la_n$ be large enough for this to be true.
	  Since  $T_\lambda$ is of general type the condition in Proposition~\ref{prop:P_{0}} is satisfied. Hence by Proposition \ref{lm:ind} there exists a non-trivial disk $D$ around $\infty$ where for any $n$ all $k$ formal solutions to the recurrence relation (\ref{eq:rec.rel}) converge to analytic functions.
	 
	Consider now $L_{n}(z)=\frac{p'_{n}(z)}{\lambda_{n} p_{n}(z)}$ for an arbitrarily fixed $n$. This is a solution to (\ref{Dlog}) and has the required zero at infinity, so there is a neighborhood of $\infty $ where 
	 $L_n(z)$  can be expanded as a convergent power series $yN_n(y)$. In particular, this power series will be one of the $k$ solutions of the recurrence relation and thus it will actually converge in $D$. It follows that the zeros
of the polynomials $p_n$ are contained in the compact region of the complex
plane that is the complement of $D$.
This fact immediately implies Theorem \ref{th:c}.

Let us now concentrate on Theorem~\ref{th:main} and start by proving the last part of this theorem. Recall that by Remark \ref{Rmk:formal} the recurrence formulas (\ref{eq:rec.rel}) converge as $\la\to \infty $ to recursive formulas determining formal solutions to the algebraic equation (\ref{eq:basic}). This implies that the formal
power series at $\infty$ for $L_{n}(z)$ converges formally (i.e., coefficientwise) to the formal power series at $\infty$ of the algebraic function $\ga(z)$ for which 
		$\lim_{z\to \infty}z\ga(z)=\alpha^{-1}$
		 (see Lemma \ref{lm:gen} in \S \ref{sec:asymp}). Since the convergence radius has a uniform non-zero lower bound, this gives that the convergence is uniform in a
neighborhood of $\infty$.
The different branches of the algebraic function $\ga(z)$ at $\infty$ correspond to the different values of $\alpha$ and this makes it clear that they all occur (cf.~Lemma~\ref{lm:gen}). This proves the last part of Theorem~\ref{th:main}.

To settle the first part of Theorem \ref{th:main}  we need some basic properties of probability measures. 
If $K$ is a compact set in $\bC$ denote by $M(K)$ the
space of all probability measures supported in $K$ equipped with the
weak topology. It is known that $M(K)$ is a sequentially compact
Hausdorff space, which allows one to choose a convergent subsequence
from any sequence of measures belonging to $M(K)$ (which is the statement of Helly's theorem).

          \begin{Lemma}[cf. Lemma 8 of \cite{BR}] Let $\{q_{m}(z)\}$ be a sequence of polynomials
with $\deg q_{m}(z)\to \infty$ as $m\to\infty$. Denote by
          $\mu_{m}$ and $\mu'_{m}$ the root-counting measures of $q_{m}(z)$
          and $q'_{m}(z)$, respectively, and assume that there exists a compact set $K$
          containing the supports of all measures $\mu_{m}$ and therefore also the
          supports of all measures $\mu'_{m}$. If $\mu_{m}\to\mu$ and
          $\mu'_{m}\to\mu'$ as $m\to\infty$ and $u$ and $u'$ are the logarithmic potentials
          of $\mu$ and $\mu'$, respectively, then $u'\le u$ in $\bC$ with 
          equality in the unbounded component of $\bC\setminus \text{supp}
          (\mu)$.
          \label{lm:hans1}
	\end{Lemma}

	\begin{Ex}\label{ex:2} Consider the polynomial sequence $\{z^m-1\}$. The measure
	    $\mu$ is then the uniform distribution on the unit circle of total
	    mass $1$. Its logarithmic potential $u(z)$ equals 
$\log\vert z\vert$
	    if $\vert z\vert \ge 1$ and $0$ in the disk $\vert z\vert \le
	    1$. On the other hand, the sequence of derivatives is given by
	    $\{mz^{m-1}\}$ and the corresponding (limiting) logarithmic potential $u'(z)$ equals $\log
	    \vert z\vert$ in  $\bC\setminus \{0\}$. Obviously,
	    $u(z)=u'(z)$ in $\vert z \vert \ge 1$ and $u'(z)<u(z)$ in $\vert
	    z\vert <1$.
\end{Ex}

In the notation of Theorem~\ref{th:main} consider the family of
eigenpolynomials $\{p_{n,j}(z)\}$ for some arbitrarily fixed value of
the index $j=1,\ldots,k$.   Assume that $N_{j}$ is a subsequence of the
natural numbers such that
\begin{equation}
\mu^{(i)}_{j}:=\lim_{n\to\infty,\,n\in N}\mu_{n,j}^{(i)}
\label{eq:mui}
\end{equation}
exists for $i=0,\ldots,k$, where $\mu_{n,j}^{(i)}$ denotes the root-counting measure of $p_{n,j}^{(i)}(z)$. The existence of such $N_{j}$ follows from the
above remark on $M(K)$ (Helly's theorem). Notice that for each $i$ the logarithmic potential
$u_{j}^{(i)}$ of $\mu_{j}^{(i)}$ satisfies a.e.~the identity
$$u_{j}^{(i)}(z)-u_{j}^{(0)}(z)=\lim_{n\to\infty,\,n\in N_{j}}\frac 1 n
\log \left\vert \frac {p_{n,j}^{(i)}(z)}{n(n-1)\ldots
(n-i+1)p_{n,j}(z)}\right\vert.$$

The next proposition completes the proof of Theorem~\ref{th:main} and also shows the remarkable property that if one considers a sequence of
eigenpolynomials for some spectral pencil then the situation
$u'(z)<u(z)$ seen in Example \ref{ex:2} can never occur.The proposition is true for a larger class of homogenized pencils of exactly solvable differential operators (\ref{eq:diff}) than those of general type. Only two properties  are needed in the proof, namely that $\deg Q_{k}(z)=k$ (i.e., $a_{k,k}\neq 0$, so that all $\alpha_j, j=1,...k$ are non-zero) and that $Q_0\neq 0$. 

\begin{Prop} The measures $\mu_{j}^{(i)}$, $i=0,\ldots,k$, are all equal and the scalar
multiple $\widetilde \Psi_{j}=C_{\mu}/\al_{j}$ of the Cauchy transform of
this common measure $\mu_{j}$ satisfies equation (\ref{eq:basic})
almost everywhere.
\label{lm:hans2}
          \end{Prop}

          \begin{proof} For $n\in N_{j}$ one has 
	$$\frac{p_{n,j}^{(i+1)}(z)}{(n-i)p_{n,j}^{(i)}(z)}\to
	C^{(i+1)}(z):=\int_{\bC}\frac{d\mu_{j}^{(i)}(\zeta)}{z-\zeta}\text{ as }n\to\infty$$
	with convergence in $L^1_{loc}$. The well-known property of
	convergence in $L^1_{loc}$ implies that passing to a subsequence one
	can assume that the above convergence is actually the pointwise
	convergence almost everywhere in $\bC$. 
	It follows that
	
\begin{equation}
\label{eq:prod }
\frac{p_{n,j}^{(i)}(z)}{n^ip_{n,j}(z)}\to C^{(1)}(z)\cdots C^{(i)}(z),
\end{equation}
	pointwise almost everywhere in $\bC$. We claim that this limit is non-zero a.e.
	Given this, consider
	$$u_{j}^{(k)}(z)-u_{j}^{(0)}(z)=\lim_{n\to\infty,\,n\in N_{j}}\frac 1 n
\log \left\vert \frac {p_{n,j}^{(k)}(z)}{n(n-1)\ldots
(n-k+1)p_{n,j}(z)}\right\vert=0$$
almost everywhere in $\bC$. On the other hand, $u^{0}_{j}\ge
u^{(1)}_{j}\ge \ldots \ge u^{(k)}_{j}$ by Lemma~\ref{lm:hans1}. Hence
the potentials
$u^{(i)}_{j}$ are all equal and the corresponding measures
$\mu_{j}^{(i)}=\Delta u^{(i)}_{j}/2\pi$ are equal as well.

It remains to show the claim.
	Recall that $p_{n,j}(z)$
	satisfies the differential equation $T_{\la_{n,j}}p_{n,j}(z)=0$,
	i.e.,
	\begin{equation}Q_{k}(z)p^{(k)}_{n,j}(z)+\la_{n,j}Q_{k-1}(z)p^{(k-1)}_{n,j}(z)+\ldots+
	\la^k_{n,j}Q_{0}(z)p_{n,j}(z)=0.
	\label{eq:start1}
	\end{equation}
	
	Therefore,
	\begin{equation*}Q_{k}(z)\frac{p^{(k)}_{n,j}(z)}{n^kp_{n,j}(z)}+\frac{\la_{n,j}}{n}Q_{k-1}(z)\frac{p^{(k-1)}_{n,j}(z)}{n^{k-1}p_{n,j}(z)}+\ldots+
	\frac{\la^k_{n,j}}{n^k}Q_{0}(z)=0.
	\label{eq:start2}
	\end{equation*}
	Using the asymptotics $\la_{n,j}\sim \al_{j}n$ and the pointwise
	convergence a.e.  in (\ref{eq:prod }) we get a.e. in $\bC$ that
\begin{equation}Q_{k}(z)C^{(1)}(z)\cdots C^{(k)}(z)+\al_jQ_{k-1}(z)C^{(1)}(z)\cdots C^{(k-1)}(z)+\ldots+
	\al_j^kQ_{0}(z)=0.
	\label{eq:start3}
	\end{equation}
	Using the hypothesis that $Q_0\neq 0\neq \alpha_j$, we conclude that $C^{(1)}(z)\neq 0$ a.e. in $\bC$.
	
	To prove that $C^{(2)}(z)$ also is non-zero a.e., we consider the differential equation satisfied by $p'_{n,j}(z)$
	
	\begin{equation*}Q_{k}(z)p^{(k+1)}_{n,j}(z)+(Q'_k+\la_{n,j}Q_{k-1}(z))p^{(k)}_{n,j}(z)+\ldots+
	(\la^{k-1}_{n,j}Q'_1(z)+\la^k_{n,j}Q_{0}(z))p{'}_{n,j}(z)=0,
	\label{eq:startder}
	\end{equation*}
 obtained by differentiation of (\ref{eq:start1}).
	Repeating the previous analysis we get
	$$
	Q_{k}(z)\frac{p^{(k+1)}_{n,j}(z)}{n^kp{'}_{n,j}(z)}+\frac{(Q'_k+\la_{n,j}Q_{k-1}(z))}{n}\frac{p^{(k)}_{n,j}(z)}{n^{k-1}p'_{n,j}(z)}+\ldots+
	\frac{(\la^{k-1}_{n,j}Q'_1(z)+\la^k_{n,j}Q_{0}(z))}{n^k}=0.
	$$
	
Hence in the limit we have
	\begin{equation*}Q_{k}(z)C^{(2)}(z)\cdots C^{(k+1)}(z)+\al_jQ_{k-1}(z)C^{(2)}(z)\cdots C^{(k)}(z)+\ldots+
	\al_j^kQ_{0}(z)=0,
	\label{eq:start32}
	\end{equation*}
	which implies that $C^{(2)}(z)$  is again non-vanishing a.e. Similarily $C^{(i)}(z), \ i\geq 3$ are a.e. non-vanishing as well.
	This proves the claim.

The fact that the multiple  $\widetilde \Psi_{j}=C_{\mu}/\al_{j}$ of the Cauchy transform of
this common measure $\mu_{j}$ satisfies equation (\ref{eq:basic})
almost everywhere follows from  (\ref{eq:start3}), since the equality of the measures implies that $C_{\mu}=C^{(1)}=C^{(2)}=...$, and so
$$
Q_{k}(z)(C_{\mu}(z))^k+\al_jQ_{k-1}(z)(C_{\mu}(z))^{k-1}+\ldots+
	\al_j^kQ_{0}(z)=0.
	$$
	This is equivalent to the desired equation. 
	\end{proof}
	
	Note that the sequence of polynomials 
	$p_n(z):=z^n-1$ in the example before the proposition, for which the conclusion of the proposition does not hold,  satisfies the differential equation $zp_n^{(2)}(z) -(n-1) p'_n=0 $, and that the corresponding homogenized pencil $z\frac{d^2}{dz^2}-n\frac{d}{dz}$ has $Q_0=0$.

The proposition finishes the proof of Theorem \ref{th:main}.  For the sake of completeness let us  prove the result on $L^1_{\rm{loc}}$-convergence used in the preceding proof. 
 
  \begin{Lemma}Suppose that $\mu_n, \ n=1,2,\ldots$, are probability measures that converge as distributions to the measure $\mu$.
  Then $z^{-1}\ast \mu_n$ converges to $z^{-1}\ast \mu $ in $L^1_{\rm{loc}}$.
  \end{Lemma}
  \begin{proof}That $z^{-1}\ast \mu_n\in L^1_{\rm{loc}}$ is a consequence of the fact that $z^{-1}\in L^1_{\rm{loc}}$ combined with Fubini's theorem (see \cite{Ga}). Write $z^{-1}=f_\epsilon +g_\epsilon$, where 
  $g_\epsilon$ is a test function with compact support and $\| f_\epsilon\|_1\leq \epsilon$. Then by Fubini's theorem one has $\| f_\epsilon\ast \mu_n\|_1\leq \epsilon$. 
  Since the functions $g_\epsilon\ast \mu_i$ converge to $g_\epsilon\ast \mu$ uniformly hence also in $L^1_{\rm{loc}}$ the lemma follows.
    \end{proof}

\section{The support of generating measures: proof of Theorems
\ref{th:rikard} and \ref{th:generic}}
\label{support}

This section is devoted to  proving Theorems
\ref{th:rikard} and \ref{th:generic}, which requires a rather elaborate mixture of ideas from algebraic geometry and complex analytic techniques. The latter are based on the main results of \cite{BB}. 

\begin{definition}\label{def-pa}
A function $\Phi$ defined in a domain $U\subset \bC$ is called 
{\em piecewise analytic}, or a $PA$-{\em function} for short, if there exists a finite family 
of pairwise distinct analytic functions $\{A_{i}(z)\}_{i=1}^{r}$ in $U$ and 
an associated family of disjoint open sets $\{M_i\}_{i=1}^{r}$ with
$M_i\subset \Omega:=\{z\in U\mid A_{k}(z)\neq A_{l}(z), 1\le k\neq l\le r\}$, 
$1\le i\le r$, such that $\{M_i\}_{i=1}^{r}$ forms a covering of $U$ up to a 
set of Lebesgue measure $0$ and 
$$\Phi(z)=\sum_{i=1}^{r}A_{i}(z)\chi_{i}(z)\text{ in }U,$$
where $\chi_{i}$ is the characteristic function of the set 
$M_{i}$, $1\le i\le r$.
\end{definition}
Note that $PA$-functions need not be continuous -- 
this will certainly  not be the case in our situation. The following result singles out the properties of the limit 
$\Phi(z):=C_{\mu_j}(z)$  that  we will need, see Theorem~\ref{th:main}. 

\begin{Prop}
\label{Prop:PA} $\Phi(z)$ is a $PA$-function in any simply-connected open set $U$. More precisely, we have that $\Phi(z)=\sum_{i=1}^{r}A_{i}(z)\chi_{i}(z)$, where $\chi_{i}$ are the characteristic functions of certain subsets of $U$ (defined in the proof below) and $A_1,\ldots,A_r$ are different (analytic) branches of the equation
$$ \sum_{i=0}^k Q_{i}(z)\alpha_j^{-i}w^i=0$$ in $U$. Moreover, 
$\mu_j=\partial \Phi (z)/\partial {\bar z}\geq 0$ as a distribution
supported in $U$. 
\end{Prop}

\begin{proof} This is an immediate corollary of Theorem ~\ref{th:main} which implies that $\alpha_j^{-1}C_{\mu_j}(z)$ satisfies an algebraic equation~(\ref{eq:basic}) almost everywhere in $\bC$. Therefore in any simply-connected  $U$  one has that $\prod_{i=1}^k(C_{\mu_j}(z)-A_i(z))=0$ almost everywhere.  Let $M_i$ (with characteristic function $\chi_{i}$) be the subset of $U$ where $C_{\mu_j}(z)=A_i(z)$. Then we immediately get that  $\Phi(z):=C_{\mu_j}(z)=\sum_{i=1}^{r}A_{i}(z)\chi_{i}(z)$  and the subsets $M_i$ cover $U$ up to a set of Lebesgue measure zero.
\end{proof}

Using some of the main results of \cite{BB} (which was initially motivated by the present project) we now derive several consequences for the support of the measure $\mu_j$ and prove Theorem \ref{th:rikard}. 
If  $\Phi(z)$ is any PA-function with $r\ge 3$  we associate to each triple $(i,j,k)$ of distinct indices in 
$\{1,\ldots,r\}$ the following set
$$\Gamma_{i,j,k}=\{z\in U\mid\exists c\in\bR\text{ such that }
A_{i}(z)-A_{k}(z)=c(A_{j}(z)-A_{k}(z))\}.$$
It is easy to see that this definition only depends on the set $\{i,j,k\}$.
Clearly, if $r\leq 2$ these sets are empty.

\begin{Theorem}[see Theorem 1 and Corollary 2 of \cite{BB}] \label{locsub}
In the above notation assume that
$\partial \Phi/\partial {\bar z}\geq 0$ as a distribution
supported in $U$ and let $p\in U$ be such that
any sufficiently small neighborhood $N$ of $p$ intersects each
$M_{i}$, $1\le i\le r$, in a set of a positive Lebesgue measure. Define 
$H_i(z)=\Re\!\left[\int _p^zA_i(w)dw\right]$ and suppose 
also that
\begin{itemize}
\item[(i)] $A_{i}(p)\neq A_{j}(p)$ for $1\le i\neq j\le r$, i.e., 
$p$ is a non-singular point of $H_{i}-H_{j}$, 
\item[(ii)] there is at most one $\Gamma_{i,j,k}$ that contains $p$.
\end{itemize}
Then there exists a neighborhood of $p$ where $\partial \Phi/\partial {\bar z}$ is supported in a union of segments of $0$-level sets to the functions $H_i-H_j$, $1\le i\neq j\le r$.
\end{Theorem}

By Proposition \ref{Prop:PA} one can apply Theorem~\ref{locsub}  to $C_{\mu_j}(z)$. 
Let $\Gamma_2$ be the set of points where conditions (i) and (ii) in the above theorem fail to hold. In other words,  $\Gamma_2$ is the set of points  $p$ which are either one of the finite number of branch points  where $A_{i}(p)=A_{k}(p)$ for a pair $(i,k)$ of distinct indices, or satisfy $A_{i}(p)-A_{k}(p)\in \bR(A_{j}(p)-A_{k}(p))$ for at least two distinct triples
$\{ i,j,k\}$ in $\{1,\ldots,r\}$. 
Notice that the
original measure $\mu_j$ can be restored
        from its Cauchy transform by
        $$\mu_j=\frac {1}{\pi}\frac {\partial C_{\mu_j}(z)}{\partial \bar
        z},$$
        where $\frac 1 \pi\frac {\partial
          C_{\mu_j}(z)}{\partial \bar z}$ is considered as a distribution, see
          e.g.~\cite[Ch.~2]{Ga}. Then Theorem~\ref{locsub} implies that in a neighborhood of any point $p\notin\Gamma_2$ the support of $\mu_j$ coincides with a  smooth $0$-level curve of some function $\Re\left[\int_{p}^z A_{i}(w)dw-\int_{p}^z A_{j}(w)dw\right]$, $i\neq j$, which  
          proves parts A and B of 
   Theorem \ref{th:rikard}.  Part C follows easily from \cite[Theorem 2]{BB}.

We turn to  proving Theorem~\ref{th:generic}, i.e., to showing that generically, the set $\Gamma_2$ is finite. The plan is to describe the 
 locus $L$ of curves such that $\Gamma_2$ is finite. Now $L$ is a semi-algebraic set, that is, $L$ is given by real polynomial inequalities and equalities. We cannot describe these explicitly since even simple cases become quite involved, as shown by Example \ref{ex:3} below.
 Instead we will prove that $L$ contains an open subset (in the Zariski hence also in the Euclidean topology) essentially by analyzing the construction of $\Gamma_{i,j,k}$ in real-algebraic geometry terms, as opposed to semi-algebraic. It will then suffice to find just one example of a curve in $L$ in order to show that $L$ is non-empty and therefore also dense, which we do in Example \ref{ex:3}.
 
Clearly, 
$\Gamma_{i,j,k}$ is either a real analytic curve or else 
there exists $c\in\bR$ such that $A_{i}(z)-A_{k}(z)=c(A_{j}(z)-A_{k}(z))$
for all $z\in U$. This latter case is certainly non-generic, and we assume that all the $\Gamma_{i,j,k}$ are real analytic curves.

First define the following variety $Y$, which  parametrizes roots of equations of degree $n$.
Set $P=\sum_{i=0}^k Q_{i} y^i\in {\bC}[Q_1,\ldots,Q_k,y]$ and let $Y\subset {\bC}^{k+1}$ be defined by $P=0$ and $Q_k\neq 0$. Denote by $p_1:  {\bC}^{k+1}\to  {\bC}^{k}$ the projection on the $k$ first coordinates  $p_1:(Q_1,Q_2,\ldots,Q_k,y)\mapsto (Q_1,Q_2,\ldots,Q_k)$ and by $p_2$ the projection on the last coordinate. The image of $Y$ under $p_1$ is the open set 
$X\subset  {\bC}^{k}$ where $Q_k\neq 0$.  The reason for imposing this last condition is to guarantee that $p_1:Y\to X$ is a finite map of affine varieties.

We can also parametrize triples of roots (so we tacitly assume that $n\geq 3$) by means of the variety $Y$. Indeed, use $p_1$ to form the fiber product $Y\times_XY\times_XY\subset  (\bC^{k}\setminus V(Q_k))\times \bC^3$ and let 
$$V:=\{ (y_1,y_2,y_3)\in Y\times_XY\times_XY | y_i\neq y_j \ {\mathrm{ if}} \ i\neq j\}.$$ 
The Zariski closure $\bar V$ is an open and closed subset of $Y\times_XY\times_XY$.
The permutation group $S_3$ acts on the last three coordinates of $\bar V\subset (\bC^{k}\setminus V(Q_k))\times \bC^3$. The orbit space $Z:=\bar V//S_3$ can be embedded as a closed subset of $ (\bC^{k}\setminus V(Q_k))\times \bC^3$ since $\bC^3//S_3\cong \bC^3$. Hence $Z$ is a quasi-affine variety. There is a natural projection from $Z$ to $X$ induced by $p_1:Y\to X$, which we denote by
 $$
p: Z\subset (Y\times_XY\times_XY)//S_3\to X.
$$
This map is finite. 
The fibre of $p$ over a point $(Q_1,\ldots,Q_k)\in X$ consists of the orbits  of
 the finite set $\{ A_{i,\alpha}\times A_{j,\alpha}\times A_{l,\alpha}\mid 1\leq i,j,l\leq k\},$ where $A_{i,\alpha},\ 1\leq i\leq k$, are the distinct solutions to the equation $P(Q,y)=0$ under permutation of coordinates.
 
 Any affine subvariety $A$ of ${\bC}^n={\bR}^{2n}$ can be considered as the set of ${\bR}$-rational points of a real algebraic variety (cf.~\cite[Ch.~3.1]{Coste-Roy}) by taking the real and imaginary parts of the complex coordinates as new coordinates. For example, ${\bC}$ itself with the ring of functions ${\bC}[z]$ is the set of ${\bR}$-rational points of the real-algebraic variety ${\bR}^{2}$ with ring of functions ${\bR}[x,y]$ and the correspondence is given by $z=x+iy$. For a closed
 subvariety $A$ of ${\bC}^n$ the corresponding ideal  of $ {\bR} [x_1,\ldots,x_n,y_1,\ldots,y_n]$ is
  $$
  I_R(A):=\langle \Re P(x+iy),\Im P(x+iy)\  | \ P\in {\bC}[z],\ P(z)=0 \ \mathrm{if} \ z\in A\rangle.
  $$ 
  Denote by $O_R(A)={\bR}[x_1,\ldots,x_n,y_1,\ldots,y_n]/I_R(A)$ the ring of real-algebraic functions on $A$.  An open subset of ${\bC}^n$ defined by the non-vanishing of a polynomial $f$ may be considered as the set of ${\bR}$-rational points of the real-algebraic variety defined by ${\bR}[x_1,\ldots,x_n,y_1,\ldots,y_n]_{\tilde f}$, where ${\tilde f}=(\Re f)^2+(\Im f)^2$. More generally, any quasi-affine subvariety in ${\bC}^n$ can  be thought of as the set of ${\bR}$-rational points of a quasi-affine real-algebraic variety. We stress this since we will need the fact that  the property that an algebraic map of complex quasi-affine varieties $A\to B$ is finite is preserved when turning $A$ and $ B$ into real algebraic varieties (in the sense of algebraic geometry, meaning that $O_R(A)$ is a finite module over $O_R(B)$). 

We now return to the  global description of  the curves $\Gamma_{i,j,k}$. By the above we can consider $\bar V$ and  its orbit space $Z$ as real algebraic varieties, and as such $Z$ contains the real algebraic subset $\tilde \Gamma$ defined by 
\begin{multline*}
\tilde \Gamma=\{(y_1,y_2,y_3)\in Y\times_XY\times_XY\mid y_i=(\alpha, z, \gamma_i)\text{ and }\\
\Im(p_2(y_1)-p_2(y_2))\overline{(p_2(y_3)-p_2(y_2))}=0\}.
\end{multline*}

\begin{not+} 
Consider the algebraic curve in ${\bC}^2$ given by  
 \begin{equation}
 \label{eq:curvegen}
P(z,y)=\sum_{i=0}^k Q_{i}(z) y^i=0,
\end{equation} 
 where $\mathrm {deg}_zQ_i\leq i$ for $i<k$ and $\mathrm {deg}_zQ_k=k$.  
Let $Q_i(z)=\sum_{0\leq j\leq i}\alpha_{ij}z^j$, $0\le i\le k$, set $\alpha=(\alpha_{ij})$ and
 denote the curve (\ref{eq:curvegen}) corresponding to $\alpha=(\alpha_{ij})$ by $L_\alpha$.
 \end{not+}

 Without loss of generality we may assume that $\alpha_{k,k}=1$. Then polynomials $P(z,y)$
 are parametrized by  the $r$-tuple $\alpha=(\alpha_{i,j})\in {\bC}^r:={\bC}^k\times{\bC}^{k}\times{\bC}^{k-1}\times \cdots\times{\bC} $. In addition we will assume that the constant polynomial  $Q_0=\alpha_{0,0}\neq 0$. In particular,
 the set $V(Q_k,Q_{k-1},\ldots,Q_0)\subset {\bC}$ where all the $Q_i$'s vanish is empty. Let
  $c_\alpha: \bC\to \bC^{k}$ be the map $c_\alpha(z)=(Q_1(z),\ldots,Q_k(z))$. Then the pullback by $c_\alpha$ of $p_1:  {\bC}^{k+1}\to  {\bC}^{k}$ is the curve $L_\alpha\subset {\bC}^2\mapsto \bC\setminus V(Q_k)$, where $V(Q_k)$ is the zero set (``support'') of $Q_k$. Hence  over a point $z\in {\bC}$ the pullback of $Y\times_XY\times_XY\to X$ will locally 
  have the finite fibre 
$$\{ (z, A_{i}(z)\times A_{j}(z)\times A_{l}(z))\vert \ 1\leq i,j,l\leq k\},$$
where $A_i,\ 1\leq i\leq k$, are the distinct branches of the algebraic curve $L_\alpha$. 
 
 Define the real-algebraic variety $\Gamma$ as the image of $\tilde \Gamma$ under the finite map of real-algebraic varieties $\bar V\to Z$. 

  \begin{Lemma}\label{lm-15} In the above notation the following holds:
\begin{itemize} 
\item[(i)]
 $\tilde \Gamma$ is preserved by the action of $S_3$. 
\item[(ii)]
 Let $\alpha=(\alpha_{i,j})$ and
 $$c_\alpha: \bC\to \bC^{r+1}$$ be the map such that $c_\alpha(z)=(Q_1(z),\ldots,Q_k(z),z)$. Then
 the pullback $c_\alpha^{-1}(p(\Gamma))\subset \bC$ is locally 
 the same as $\bigcup\Gamma_{i,j,k}$, where the union is taken over all indices 
$i,j,l$ for the curve $L_\alpha$.
\end{itemize}
 \end{Lemma}
\begin{proof}
The identity $\Im(u-w)\overline{(v-w)}=\Im(u\bar v-u\bar w-w\bar v)$ gives  the first 
part. Now $\Im(u-w)\overline{(v-w)}=0$ if and only if $\frac{u-w}{v-w}\in \bR$, which together with the definitions yields statement (ii) in the lemma.
\end{proof}

Let $R_\Gamma$ be the $\bR$-algebra whose subset of $\bR$-rational points coincides with $\Gamma$ and let $R_X$ be the $\bR$-algebra that corresponds to $X$ in similar fashion. Then the map $p$ corresponds to the map $p^*:R_X\to R_\Gamma$ and the kernel $J$ of $p^*$ defines the image of $p$. Furthermore, since $p^*$ is finite the $R_X$-annihilator $J_2\supset J$ of $R_\Gamma/p^*(R_X)$, that is, 
$$J_2=\{r\in R_X \ | \ p(r)s=0,\ s\in R_\Gamma\}$$ 
exists and defines the {\em support} $V(J_2)$ of $R_\Gamma/p^*(R_X)$. It also defines   the locus of prime ideals $m$ in $\mathrm {Spec}\;  R_X$ such that the fibre $R_\Gamma/p^*(m)$ is not equal to $R_X/m$. Thus $V(J_2)$ defines the locus where the fibre has multiplicity (strictly) greater  than one. (The existence of such an ideal was the whole point of our excursion into  algebraic geometry.) Note that the fact that the multiplicity of the fibre over a $\bR$-rational point in $X$ is greater than one does not guarantee the existence of  $\bR$-rational points in the fibre.

Now we can finally reduce the proof of Theorem \ref{th:generic} to finding an explicit example. Consider the map
$c: {\bC}^r\times {\bC}\to {\bC}^k$ defined by $c(\alpha,z)=(Q_1(z),\ldots,Q_k(z))$. (The connection between $\alpha$ and $Q_i$ is as above.) Let $A={\bC}^r\times {\bC}\setminus V(Q_k)$.  Take the map $c^*:R_X\to R_A$ corresponding to $c$. The ideal $c^*(I_2)R_A$ describes the points in $\mathrm{Spec}\; R_A$ such that the fibre of the pullback has multiplicity greater than one. Denote the real coordinates in ${\bC}^r\times {\bC}$ by $x_{st}$, $y_{st}$, $x$, $y$ such that $\alpha_{st}=x_{st}+iy_{st}$ and $z=x+iy$. A point $\alpha=(\alpha_{st})\in {\bC}^r$ corresponds to a maximal ideal $m_{\alpha}$. By Lemma \ref{lm-15} 
the ring $$
S_{\alpha}:={\bR}[x,y]/m_{\alpha}c^*(I_2)R_A
$$ describes the set of points in $\bC $ that lie on two distinct $\Gamma_{i,j,l}$'s.

The following lemma (whose simple proof we included for completeness) proves that the property of $\alpha$ that $S_{\alpha}$ is finite as a $\bR$-vector space is  generic.

  \begin{Lemma}\label{lem-16}
 Suppose that $c: T\to S$ is a map of finitely generated $\bR$-algebras and let $U$ be the set of $\bR$-rational points $m$ in $\mathrm {Spec}\; T$ such that $S/c(m)$ is a finite $\bR$-algebra. Then $U$ is open. 
  \end{Lemma}

\begin{proof}
Let $V$ be a finite-dimensional vector space containing $1\in S$ that generates $S$, and let $V^n$ denote the set of sums of products of at most $n$ elements from $V$ with coefficients in $c(T)$.  The fact that  $V$ generates $S$ implies that $S=\bigcup_{n\ge 0} V^n$.
Now $S/c(m)$ is a finite $\bR$-algebra if and only if there exists a non-negative integer   $n$ such that 
$$T/m\otimes_T(V^{n+1}/V^n)=0.
$$
Let $J_k:=\mathrm{Ann}_T\ V^{k+1}/V^k$, $k\ge 0$. Since $T$ is a ring one has that $J_k\subset J_{k+1}$. Set $J=\bigcup_{k\ge 0} J_k$. Then $U$ consists of all $\bR$-rational points $m$ in $\mathrm {Spec}\; T$ which are not contained in the support $V(J)$. Since the latter is obviously closed this proves the lemma. 
\end{proof}

Applied to  our situation  this means that there exists  a set of algebraic equations in the variables $x_{st}$, $y_{st}$
which when violated will guarantee that the number of points in $\bC$ that lie on two distinct $\Gamma_{i,j,l}$'s is finite. Thus to 
conclude the proof  of Theorem \ref{th:generic} it only remains to find an example, which we do next.

\begin{Ex}\label{ex:3}
Let $b_1,\ldots,b_k$ be complex numbers. Consider the reducible plane curve $L$  given by 
$$
\prod_{i=1}^k((b_i-z)y-1)=0.
$$
This equation has the solutions $y=\frac{1}{b_i-z}$, $i=1,\ldots,k$, and it satisfies the conditions of (\ref{eq:curvegen}).
To describe the sets $\Gamma_{i,j,k}$ we use the following notation. 
Let $\alpha_i$, $\beta_i$, $i=1,2,3$, be our variables and define an action of the symmetric group $S_3$  on polynomial functions in these variables by permuting the pairs $(\alpha_i,\beta_i)$. If $M$ is a monomial let $\eta_M$ denote the alternating monomial function 
$$\eta_M=\sum_{\sigma\in S_3}(-1)^{sign(\sigma)}\sigma M.
$$
(For example, $\eta_{\alpha_1\beta_2}=\alpha_1\beta_2-\alpha_2\beta_1-\alpha_1\beta_3+\alpha_2\beta_3+\alpha_3\beta_1-\alpha_3\beta_2$.) 
 The set $\Gamma$ for the curve $L$ is reducible and consists of one component for each $\{i,j,k\}$. Simple calculations show that the set $\Gamma_{1,2,3}$ is  a circle in $ \bC$ consisting of all points $z=u+iv$ satisfying the equation
\begin{multline*}
p_{1,2,3}=-\eta_{\alpha_1\beta_2}(u^2+v^2)+(\eta_{\alpha_1\beta_2^2}+\eta_{\alpha_1\alpha_2^2})v
\\
+(\eta_{\beta_1^2\beta_2}+\eta_{\alpha_1^2\beta_2})u
+(\eta_{\alpha_1\beta_2\beta_3^2}+\eta_{\alpha_1\alpha_3^2\beta_2})=0. 
\end{multline*}
The center of this circle is given by 
$$-\left(\frac{\eta_{\beta_1^2\beta_2}+\eta_{\alpha_1^2\beta_2}}{2\eta_{\alpha_1\beta_2}},
\frac{\eta_{\alpha_1\beta_2^2}+\eta_{\alpha_1\alpha_2^2}}{2\eta_{\alpha_1\beta_2}}\right),$$ 
where it is understood that $b_j=\alpha_j+i\beta_j$, $j=1,2,3$, and that $\Gamma_{1,2,3}$ is non-empty (which is generically true). Hence for a generic set of complex numbers $b_1,\ldots,b_k$ all sets $\Gamma_{i,j,k}$ are circles with different centers, so that in particular the intersection of two sets $\Gamma_{i,j,k}$ corresponding to two different index sets $\{i,j,k\}$ is finite. $S_\alpha$ is the direct sum of the algebras $\bR[u,v]/I$, where $I$ is the ideal generated by two different quadrics $p_{i,j,k}$ and $p_{r,s,t}$, and it is therefore finite. This shows that the set $U$ in Lemma \ref{lem-16} is non-empty. \end{Ex}

The proof of Theorem  \ref{th:generic} is now complete.

\section {Final remarks and problems}
\label{sec:problems}
             \medskip
             
             The main topic of the paper is intimately related to the classical
   asymptotic theory of linear ordinary differential equations and 
   with the WKB-method. These  are covered in a huge
bulk of both older and modern literature including the classical
books by J.~Ecalle \cite {Ec}, A.~Erd\'elyi \cite {Er}, M.~Fedoryuk \cite {Fe},
B.~Malgrange \cite {Ma}, F.~Olver \cite {Ol}, Y.~Sibuya \cite {Sib},
W.~Wasow \cite {Wa1}-\cite {Wa2}
as well as important contributions by T.~Aoki, Y.~Takei, 
T.~Kawai, M.~Berry,
E.~Delabaere,   F.~Pham, Y.~Colin de Verdi\`ere, A.~Voros and many
others, see e.g.~\cite {Be}, \cite{BM}, \cite{De}, \cite {P}, \cite {Vo}.
   Polynomial solutions to linear differential equations are also of special
interest in connection with the so-called quasi-exactly solvable
models in quantum mechanics, see papers by A.~Turbiner, e.g.~\cite {Tu} and
references therein. They are also the classical object of study in the
theory of orthogonal polynomials which recently acquired  new
effective tools to investigate the asymptotical properties of families of such
solutions, see \cite {DZ}.

The starting point of our study  was a number of rather surprising numerical
experiments in calculating the zeros of polynomial eigenfunctions for the
usual spectral problem which led to several  conjectures
presented in \cite {MS}.  Most of these
conjectures were later proved in \cite {BR}. The idea to consider the
homogenized spectral problem comes mainly from  \cite {Fe}, \cite {Ol} and
references \cite {AKT1}, \cite {AKT2}. 
      
          By carefully going through the proofs of the above theorems on can see that actually these apply to a wider class of $ES$-pencils. Namely, consider $ES$-pencils of the form
$T_{\la}=\sum_{i=0}^k Q_{i}(z,\lambda)
\frac {d^i}{dz^i}$, 
where $Q_{i}(z,\lambda)$ are polynomials in $z$ and $\la$. Assume that
${\mathrm{deg}}_{\lambda}Q_{i}(z,\lambda)\leq k-i$ and that ${\mathrm{deg}}_{z}Q_{i}(z,\lambda)\leq i$.  Decompose $Q_i(z,\la)=Q_{i,0}(z)+\ldots+Q_{i,k-i}(z)\la^{k-i}$. If the homogenized pencil $\tilde T_\la=\sum_{i=0}^k Q_{i,k-i}(z)\la^{k-i}\frac {d^i}{dz^i}$ is of general type then all the results in this paper apply. However, for the sake of brevity we have chosen to focus only on the homogeneous case.


Let us briefly formulate a number of related problems.

\medskip
\noindent 
{\bf 1. Stokes lines.}
              One of the major objects in the classical
              asymptotic theory is the {\em  global
              Stokes line}  of a given linear differential equation, see e.g.~\cite {AKT1}, \cite {AKT2}.
The theory of {\em global Stokes lines}	    is  well understood in the physically important case of second  order linear odes.  By contrast with the latter case the development of
	    this theory for higher order equations experiences serious
	    difficulties. These are mostly due  to the 
appearance of new
              Stokes lines (discovered in \cite {BNR}) originating from the intersection points of the
              initial  Stokes lines which emanate from the turning points of the
              considered equation. 
	    The supports of the asymptotic measures $\mu_{j}$
	    studied in the present paper show
              many similar features to the global Stokes line of a linear
              differential operator (\ref{eq:diff}) which leads to the
              following conjecture.

     \begin{Conj}
The union  of the supports of the
              measures $\mu_{1},\ldots,\mu_{k}$ is a subset of the
              global Stokes line for the  differential operator (\ref{eq:diff}).

           \end{Conj}

\medskip 
\noindent
{\bf 2. Convergence and topology of supports.} 
The main open questions related to the results of the present paper are Conjecture~1 (see the introduction) 
and a complete description of the topological properties of the support of the limiting measures $\mu_j$, $j=1,\ldots.,k$. (We strongly believe that $\mu_j$ is unique for each $j$.) In particular, the following problem appears to be of fundamental importance and all the numerical evidence we obtained so far is in favor of this conjecture. 

\begin{Conj} 
For each $j$ the set $\bC\setminus supp(\mu_j)$ is connected. 
\end{Conj} 

\begin{Prob}
Is it true that the support of each $\mu_j$ is topologically a tree?
\end{Prob}

In general, it would be highly desirable to describe the stratification of the space $ES_k$ according to the topology of the supports of the $k$-tuple of measures $\mu_1,\ldots,\mu_k$.  

\medskip
\noindent
{\bf 3. Other classes of operators.} A further natural question is as follows.

\begin{Prob}
Which of the above results hold for operators
(\ref{eq:diff}) of nongeneral type?
\end{Prob}

Numerical experiments carried out in \cite{Berg} show the existence of
asymptotic root-counting measures even for such operators. However, these
measures usually have noncompact supports, which leads to considerable 
additional complications since the arguments in all our proofs do not 
extend {\em mutatis mutandis} to such situations.

	\medskip
	\noindent 
{\bf 4. Positive Cauchy transforms.} As we already mentioned in the introduction, Problem~\ref{pb2} is completely solved only for rational functions. 

\begin{Prob}
Characterize algebraic germs of real (respectively, positive) Cauchy transforms.
\end{Prob}

\end{document}